\documentclass[11pt, reqno]{amsart}

\usepackage{graphics, stackrel}
\usepackage{amsmath, amssymb, amsthm}
\usepackage{graphicx}
\usepackage{verbatim}
\usepackage{amsfonts}
\usepackage{natbib}
\usepackage{enumitem}
\usepackage[T1]{fontenc}
\usepackage{fontspec} 
\usepackage[xcharter]{newtxmath}
\usepackage{float}

\usepackage{ulem}
\usepackage{endnotes}
\usepackage{fancyvrb}
\usepackage[svgnames,table]{xcolor}

\usepackage[citecolor=blue, colorlinks=true, linkcolor=blue]{hyperref}

\usepackage{enumitem}
\setlist[enumerate]{itemsep=2pt,topsep=3pt}
\setlist[itemize]{itemsep=2pt,topsep=3pt}
\setlist[enumerate,1]{label={\upshape (\roman*)}}

\usepackage{mathrsfs}  % caligraphic
\usepackage{bm, bbm}   % bold symbols

\usepackage[left=1.0in, right=1.0in, top=1.0in, bottom=1.0in, includehead, includefoot]{geometry}

\renewcommand{\leq}{\leqslant}
\renewcommand{\geq}{\geqslant}

\usepackage[ruled, linesnumbered]{algorithm2e}

\DeclareMathOperator*{\argmax}{arg\,max}

\newcommand{\too}{\stackrel { o } {\to} }

\newcommand{\setntn}[2]{ \{ #1 : #2 \} }
\newcommand{\fore}{\therefore \quad}

\newcommand{\1}{\mathbbm 1}

\newcommand*\diff{\mathop{}\!\mathrm{d}}

\newcommand{\aA}{\mathscr A}
\newcommand{\bB}{\mathscr B}

\newcommand{\sS}{\mathscr S}

\newcommand{\Asf}{\mathsf A}
\newcommand{\Bsf}{\mathsf B}

\newcommand{\Ssf}{\mathsf S}

\newcommand{\Gsf}{\mathsf G}

\newcommand{\Xsf}{\mathsf X}
\newcommand{\Ysf}{\mathsf Y}
\newcommand{\Zsf}{\mathsf Z}

\newcommand{\RR}{\mathbbm R}

\newcommand{\NN}{\mathbbm N}

\newcommand{\TT}{\mathbbm T}

\newcommand{\EE}{\mathbbm E}

\renewcommand{\phi}{\varphi}
\renewcommand{\epsilon}{\varepsilon}

\theoremstyle{plain}
\newtheorem{theorem}{Theorem}[section]

\newtheorem{lemma}[theorem]{Lemma}
\newtheorem{proposition}[theorem]{Proposition}

\theoremstyle{definition}

\newtheorem{example}{Example}[section]
\newtheorem{remark}{Remark}[section]

\newtheorem{assumption}{Assumption}[section]

\newcommand{\navy}[1]{\textbf{\textcolor{Brown}{#1}}}

\usepackage{tikz}
\usetikzlibrary{arrows.meta, positioning, fit}
\begin{document}

\title[Dynamic Programming]{Dynamic Programming in Ordered Vector Space}

\author{John Stachurski}
\address{National Graduate Institute for Policy Studies (GRIPS), Tokyo, Japan}
\email{john.stachurski@gmail.com}

\author{Nisha Peng}
\address{Research School of Economics, Australian National University, Canberra, Australia}
\email{chengyuanp2022@gmail.com}

\begin{abstract}
    \vspace{1em}
    New approaches to the theory of dynamic programming view dynamic
    programs as families of policy operators acting on partially ordered sets.
    In this paper, we extend these ideas 
    by shifting from arbitrary partially ordered sets to ordered vector
    spaces. The integrated algebraic and order structure in such spaces leads to sharper
    fixed point results.  These fixed point results can then be exploited to
    obtain optimality properties.  We illustrate our results
    through applications ranging from firm management to data valuation.
    These applications include features from the recent literature on dynamic programming, including
    risk-sensitive preferences, nonlinear discounting, and state-dependent
    discounting. In all cases we establish existence of optimal policies,
    characterize them in terms of Bellman optimality relationships, and prove
    convergence of major algorithms.


    \smallskip
    \noindent \textbf{Subject classification.} 
    \emph{dynamic programming, Markov, utility}

    \smallskip
    \noindent \textbf{Area of review.} 
    \emph{Decision Analysis}
\end{abstract}

\maketitle

\section{Introduction}

Dynamic programming is a core component of modern optimization theory, with
application domains ranging from operations research to finance, economics, and
biology (see, e.g., \cite{bellman1957dynamic, bauerle2011markov, bertsekas2012dynamic,
bloise2023dont}).
In recent years, dynamic
programming has formed an important part of artificial intelligence theory and
practice, mainly through reinforcement learning \citep{bertsekas2021rollout,
kochenderfer2022algorithms, murphy2024reinforcement}.  

One of the most striking features of the current dynamic programming landscape is the huge
variety of applications.  This variety comes from two sources.  The first is the
ever-increasing range of use cases, as new problems come to 
the attention of
researchers (e.g., asteroid belt exploration, user engagement maximization on
social media platforms, etc.). The second
is that researchers are taking traditional applications (e.g., optimal savings
and investment problems) and modifying the
specification of the dynamic program to include new features, such as
risk-sensitivity, nonlinear recursive preferences, nonlinear discounting,
ambiguity aversion, desire for robustness, and the use of quantiles or
distributions instead of expectations.  (For a small selection of this 
literature, see \cite{maccheroni2006dynamic}, \cite{marinacci2010unique},
\cite{de2019dynamic}, \cite{jaskiewicz2014variable},
\cite{bauerle2021stochastic}, or  \cite{TE20250353}.)

Handling such a vast range of potentially diverse applications within a single unified
framework requires high-level theory and significant abstraction, 
stripping away specific features of particular problems in order to reduce dynamic 
programming to its essence.  This has led to the development of
``abstract dynamic programming,'' led by \cite{bertsekas2022abstract} and
further developed by numerous authors, including \cite{stachurski2021dynamic}, 
\cite{bloise2023dont}, \cite{toda2024unbounded}, and
\cite{sargent2025partially}. For example, the theory developed in
\cite{sargent2025partially} provides an abstract framework where dynamic programs
are represented as families of policy operators acting on a partially ordered
set.  By taking this very high level of abstraction, the authors are able to provide
optimality and algorithmic convergence results that apply to a vast range of
applications.

One disadvantage of the results in \cite{sargent2025partially} is that their
high-level structure makes it difficult to
test their conditions in specific applications.  In this paper, we begin to rectify this
deficiency by adding bridging results that are
 significantly easier to test. We do this by
developing and exploiting fixed point results in ordered vector space.  
For example, we use a recent result
of \cite{marinacci2019unique}, which provides a sufficient condition for
the uniqueness of fixed points of monotone concave operators. We also establish some
new fixed point results in Riesz space, which we then exploit to generate
additional
optimality and convergence results for abstract dynamic programming theory.  

The value-added provided by our results is illustrated in our applications. We
use these applications to show that both standard and exotic versions of dynamic
programming can be handled using the methods developed in this paper.  In
addition, we use the same methods to establish new optimality and algorithmic
convergence results for specific applications. These applications pertain to
\begin{enumerate}
    \item risk-sensitive preferences \citep{bauerle_markov_2022,
        bauerle_markov_2024, bauerle_stochastic_2018}, where we extend the
        risk-sensitive framework (which is also connected to robust control) to
        handle state-dependent discounting,
    \item quantile preferences \citep{giovannetti2013asset, de2019dynamic, 
        de2022static, TE20250353},
        where we investigate a Q-factor formulation of Markov decision processes
        with quantile preferences, 
        opening up the possibility of using Q-learning-based reinforcement
        learning with this class of certainty equivalents, 
        and showing that both value function iteration and
        policy iteration algorithms converge, 
        and
    \item nonlinear discounting \citep{jaskiewicz2014variable,
        bauerle2021stochastic}, where we consider a variation on the model that
        can potentially include features such as state-dependent discounting or
        negative interest rates (i.e., discount factors that are occasionally
        greater than one).
\end{enumerate}

In all cases, we characterize optimal policies
and prove convergence of major algorithms.  Throughout the discussion of
applications, we emphasize how the framework we develop
has direct implications for practical problems. For instance, in firm management,
we provide a numerical study to show how the results help shed light on firm
valuation and choices in the context of risk-sensitive
responses to stochastic cash flows.  We show that the analysis remains
straightforward even under the realistic assumption that discounting varies with the economic
environment.  This allows a firm's behavior to
be traced under different risk-aversion parameters and multiple sources of
risk. With minor modifications, the same ideas can be applied to a wide range of dynamic optimization problems, such as consumption–savings, inventory management, portfolio selection, and job search.

In addition, to illustrate one of our fixed point results, we
also discuss existence of solutions to the valuation of firm-specific digital
data discussed in \cite{veldkamp2023valuing}.

Under the hood, we use order completeness properties of certain ordered vector
spaces to establish sharp results. The key idea is as follows.
\cite{sargent2025partially} shows that strong optimality results for dynamic programs
require that policy operators have unique fixed points and the Bellman operator has at
least one fixed point.  A stumbling block is that while the policy operators
might have nice properties that allow us to obtain necessary fixed point
results, these properties might not transfer over to the Bellman operator.
(For example, the concavity properties used in the fixed point result from
\cite{marinacci2019unique} might hold for the policy operators but fail for the
Bellman operator, since the Bellman operator involves a supremum, and this
supremum operation breaks concavity.)  Here, we use the fact that
many ordered vector spaces have order completeness properties, while the Bellman 
operator naturally has order preserving properties.  These features can be
combined to obtain a fixed point for the Bellman operator.

The structure of the paper is as follows.  In Section~\ref{s:pre}, we
present some mathematical notations. In Section~\ref{s:adp}, we briefly recall the
dynamic programming framework of \cite{sargent2025partially}.  
Section~\ref{s:fprs} provides several fixed-point results in Riesz
spaces.  Our main optimality results are in Section~\ref{s:opti}.
Section~\ref{s:appl} contains applications.

\section{Preliminaries}\label{s:pre}

This introductory section reviews necessary mathematical concepts.  
The discussion of partially ordered sets and
ordered vector space largely follows \cite{zaanen2012introduction}.  
Throughout, when $A$ and $B$ are sets, the symbol $B^A$ represents the set of
maps from $A$ to $B$.

\subsection{Posets}

Let $V = (V, \preceq)$  be a partially ordered set, also called a poset. If $v
\preceq w$ are two elements in $V$, then $[v, w]$ represents the set $\{u \in V:
v \preceq u \preceq w\}$. Each such set is called an \navy{order interval} in
$V$. A subset $A$ of $V$ is called \navy{order bounded} if there exists an order
interval $I \subset V$ with $A \subset I$.  In line with this definition, the
set $V$ is order bounded if $V$ itself can be expressed as an order interval.
A \textbf{chain} is a partially ordered set such that every two
elements $a$ and $b$ satisfy $a \preceq b$ or $b \preceq a$ (or both).
A poset $V$ is called \navy{chain complete} (resp., \navy{countably chain
complete}) if $V$ is order bounded and every chain (resp., every countable chain) 
has a supremum and an infimum in $V$.  The poset $V$ is called \navy{Dedekind complete} 
(resp., \navy{countably Dedekind complete}) if every order interval is chain complete
(resp., countably chain complete).  

A sequence $(v_n)$ in $V$ is called  \navy{increasing} if $v_n \preceq v_{n+1}$ for all $n
\in \NN$.  If $(v_n)$ is increasing and $\bigvee_n v_n = v$ for some $v \in E$, then we
write $v_n \uparrow v$.  (Here and below, $\bigvee$ denotes a supremum and $\bigwedge$
denotes an infimum.)  Similarly, we write $v_n \downarrow v$ if $(v_n)$ is decreasing and
$\bigwedge_n v_n = v$ for some $v \in V$.  
A self-map $S$ on $V$ is called \textbf{order preserving} if
$v \preceq w$ implies $S v \preceq S w$.
A self-map $S$ on $V$ is called \navy{order continuous} on $V$
if $Sv_n \uparrow Sv$ whenever $v_n \uparrow v$.  In other words, 
if $(v_n) \subset V$ with $v_n \uparrow v \in V$, then supremum $\bigvee_n Sv_n$ exists in
$V$ and equals $Sv$.   It follows from the definition that every order continuous self-map
on $V$ is also order preserving.

\subsection{Ordered Vector Space}\label{ss:rs}

We recall that an ordered vector space is a vector space 
over real field $E$ paired with a
partial order $\leq$ satisfying $u \leq v$ implies $u + z \leq v + z$ for all $z
\in E$ and $r u \leq r v$ for all nonnegative $r \in \RR$. The positive cone
$E_+$ of $E$ is all $v \in E$ with $0 \leq v$. For $v, v' \in E$, the statement
$v < v'$ means that $v \leq v'$ and $v$ and $v'$ are distinct. 
A positive linear operator on $E$ is a linear map $L \colon E \to E$ with $L E_+
\subset E_+$.

A Riesz space is an ordered vector space $E = (E, \leq)$ that is also a lattice
(i.e., closed under finite suprema and infima).   The absolute value of $u \in
E$ is defined by $|u| \coloneq (u^+) \vee (u^-)$ where $u^+ \coloneq u \vee 0$
and $u^- \coloneq (-u) \vee 0$.  Regarding absolute value, we will use the
following simple facts:
\begin{itemize}
    \item For any $v \in E$ and $u \in E_+$, we have $|v| \leq u \iff v \in [-u, u]$.
    \item If $K \colon E \to E$ is a positive linear operator, then $|Ku| \leq K |u|$ for
    all $u \in E$.
\end{itemize}

A (not necessarily monotone) sequence $(v_n)$ in a Riesz space $E$ is said to
\navy{order converge} to a point $v \in E$ if there exists a sequence $d_n
\downarrow 0$ in $E$ with $|v_n - v| \leq d_n$ for all $n \in \NN$.  In
this case we write $v_n \too v$. Order limits are unique.   Moreover, if $(v_n)$ is
increasing, then $v_n
\too v$ if and only if $v_n \uparrow v$.  An analogous result holds for
decreasing sequences.  (See, e.g., \cite{zaanen2012introduction}.)

When $(\Xsf, \aA, \mu)$ is a $\sigma$-finite measure space, we let 
\begin{itemize}
    \item $\RR^{\Xsf}$ be the real-valued functions on $\Xsf$,
    \item $m\Xsf$ be the real-valued Borel measurable functions on $(\Xsf,
        \aA)$, and
    \item $bm\Xsf$ be the bounded functions in $m\Xsf$. 
\end{itemize}
The spaces $\RR^{\Xsf}$, $m\Xsf$, and $bm\Xsf$ are all countably
Dedekind complete Riesz spaces under the
pointwise partial order $\leq$.  The space $\RR^{\Xsf}$ is also Dedekind complete.

\subsection{Markov Dynamics}

If $(\Xsf, \bB)$ is a measurable space and $U$ is a topological space, a
\navy{stochastic kernel} $P$ from $U$ to $\Xsf$ is a function $P \colon U \times
\bB \to \RR_+$ with the property that $u \mapsto P(u, B)$ is Borel measurable
for each $B \in \bB$ and $B \mapsto P(u, B)$ is a probability on $(\Xsf, \bB)$
for all $u \in U$.  When $U = \Xsf$, we say that $P$ is a stochastic kernel on
$\Xsf$.  Given a stochastic kernel $P$ from $U$ to $\Xsf$, we set 
\begin{equation*}
    (P h)(u) \coloneq \int h(x') P(u, \diff{x'})
    \qquad (h \in bm\Xsf, \; u \in U).
\end{equation*}

\section{Abstract Dynamic Programs}\label{s:adp}

The purpose of this section is to 
briefly review the dynamic programming framework from
\cite{sargent2025partially} 
in order to make the paper partly self-contained.

\subsection{Definition and Properties}\label{ss:adp}

An \navy{abstract dynamic program} (ADP) is a pair $(V, \TT)$, where 
\begin{enumerate}
    \item $V = (V, \preceq)$ is a poset and
    \item $\TT = \setntn{T_\sigma}{\sigma \in \Sigma}$ is a nonempty family of
        order preserving self-maps on $V$.
\end{enumerate}
The set $V$ is called the \navy{value space}.  Each $T_\sigma \in \TT$ is called a 
\navy{policy operator}.  $\Sigma$ is an index set and each $\sigma \in \Sigma$ is called
a \navy{policy}.  In applications, we impose conditions under which each $T_\sigma$
has a unique fixed point.  In these settings, the significance of each policy
operator $T_\sigma$ is that its fixed point, denoted below by $v_\sigma$, represents the
lifetime value of following policy $\sigma$.  

\begin{example}
    Consider a Markov decision process with finite state space $\Xsf$ and finite
    action space $\Asf$ (see \cite{puterman2005markov}).
    The set of feasible actions
    at state $x$ is denoted by $\Gamma(x)$, where $\Gamma$ is a nonempty
    correspondence from $\Xsf$ to $\Asf$.
    The set $\Gsf \coloneq \setntn{(x, a) \in \Xsf \times \Asf}{a \in
    \Gamma(x)}$ denotes the feasible state-action pairs and $r$ is a reward
    function mapping $\Gsf$ to $\RR$. Letting $\beta \in (0,1)$ be a discount factor
    and $P \colon \Gsf \times \Xsf \to [0,1]$ provide transition probabilities,
    the Bellman equation  is 
    \begin{equation}\label{eq:mdp_bell}
        v(x) = \max_{a \in \Gamma(x)} 
        \left\{
            r(x, a) + \beta \sum_{x'} v(x') P(x, a, x')
        \right\}
        \qquad\qquad (x \in \Xsf).
    \end{equation}
    The set of feasible policies is the finite set $\Sigma \coloneq \setntn{\sigma
    \in \Asf^\Xsf} {\sigma(x) \in \Gamma(x) \text{ for all } x \in \Xsf}$. We
    combine $\RR^\Xsf$ (the set of all real-valued functions on $\Xsf$) with the
    pointwise partial order $\leq$ and, for $\sigma \in \Sigma$ and $v \in
    \RR^\Xsf$, define the MDP policy operator
    \begin{equation}\label{eq:tsig_mdp}
        (T_\sigma \, v)(x) 
            = r(x, \sigma(x)) + \beta \sum_{x'} v(x') P(x, \sigma(x), x')
            \qquad\qquad (x \in \Xsf).
    \end{equation}
    With $\TT \coloneq \setntn{T_\sigma}{\sigma \in \Sigma}$, the pair $(V,
    \TT)$ is an ADP.  If $v_\sigma$ is the (unique) fixed point of $T_\sigma$,
    then $v_\sigma(x)$ represents the lifetime value of using policy
    $\sigma$ in every period, conditional on starting in state $x$ (see, e.g.,
    \cite{sargent2025partially}, \S3.1, or \cite{puterman2005markov}).
\end{example}

Let $(V, \TT)$ be an ADP with policy set $\Sigma$. Given $v \in V$, a policy
$\sigma \in \Sigma$  is called \navy{$v$-greedy} if $T_\tau \, v \preceq T_\sigma \, v$
for all $\tau \in \Sigma$. 
An element $v \in V$ is said to satisfy the \navy{Bellman equation} if 
\begin{equation}\label{eq:adp_belleq}
    v \coloneq \bigvee_{\sigma} T_\sigma \, v
    \qquad (v \in V).
\end{equation}
In \eqref{eq:adp_belleq}, the supremum is taken over all $\sigma \in \Sigma$.
The \navy{Bellman operator} generated by $(V, \TT)$  is defined via
\begin{equation*}
    Tv \coloneq \bigvee_{\sigma} T_\sigma \, v
    \quad \text{whenever the supremum exists}.    
\end{equation*}
Evidently $v \in V$ satisfies the Bellman equation if and only if $Tv$ exists and $Tv = v$.
Also, $T_\sigma \, v \preceq Tv$ for all $\sigma \in \Sigma$ and
$T_\sigma \, v = T v$ if and only if $\sigma \in \Sigma$  is $v$-greedy.

We call $(V, \TT)$ 
\begin{itemize}
    \item \navy{well-posed} if each $T_\sigma \in \TT$ has a unique fixed point in $V$,
        and
    \item \navy{regular} if a $v$-greedy policy exists for all $v \in V$.
\end{itemize}
Well-posedness is a minimal condition because without it we cannot be sure that
policies have well-defined lifetime values. Maximizing lifetime values
(or, equivalently, minimizing lifetime costs) is the objective of all dynamic
programming problems. Regularity is useful for the existence of optimal policies
(defined below).

\subsection{Optimality}

When $(V, \TT)$ is well-posed, we let $V_\Sigma$ denote all
elements of $V$ that are fixed points for some $T_\sigma \in \TT$.
We say that a policy $\sigma \in \Sigma$ is \navy{optimal} for
$(V, \TT)$ if $v_\sigma$ is a greatest element of $V_\Sigma$.
In other words, $\sigma$ is optimal if it attains the ``highest possible
lifetime value.''

Let $(V, \TT)$ be a well-posed ADP and set
\begin{equation*}
    v^* \coloneq \bigvee_\sigma v_\sigma \coloneq \bigvee V_\Sigma
    \quad \text{whenever the supremum exists}.
\end{equation*}
When $v^*$ exists (i.e., when the supremum exists) we call $v^*$ the \navy{value function}
of the ADP. \endnote{Admittedly, $v^*$ is not necessarily a function in our abstract
framework. The terminology is chosen to align objects with
their counterparts in traditional dynamic programming theory.}
The following statements are obvious from the definitions:
\begin{itemize}
    \item Existence of an optimal policy $\sigma$ implies that $v^*$ exists and is equal to
    $v_\sigma$.
    \item If $v^* = v_\sigma$ for some $\sigma \in \Sigma$, then $\sigma$ is
        optimal.
\end{itemize}
\navy{Bellman's principle of optimality} is said to hold if $(V, \TT)$ is well-posed and
\begin{equation*}
  \setntn{\sigma \in \Sigma}{\text{$\sigma$ is optimal }}
  = 
  \setntn{\sigma \in \Sigma}{\text{$\sigma$ is $v^*$-greedy }}.
\end{equation*}
(When $v^*$ does not exist, both sets are understood as empty.)
For a well-posed and regular ADP $(V, \TT)$, the \navy{fundamental optimality
properties} are said to hold when the following statements are true:
\begin{enumerate}\label{enum:b13}
    \item[(B1)] at least one optimal policy exists,
    \item[(B2)] $v^*$ is the unique solution to the Bellman equation in $V$, and
    \item[(B3)] Bellman's principle of optimality holds. 
\end{enumerate}

\subsection{Algorithms}

For a regular, well-posed ADP $(V, \TT)$, the \navy{Howard policy operator}
corresponding to $(V, \TT)$ is defined via
\begin{equation*}
    H \colon V \to V_\Sigma, \qquad
    H v = v_\sigma \quad \text{ where $\sigma$ is $v$-greedy}.
\end{equation*}
For each $m \in \NN$, the \navy{optimistic policy operator} is defined by 
\begin{equation*}
    W_m \colon V \to V, \qquad
    W_m v 
    = T^m_\sigma v \quad \text{where $\sigma$ is $v$-greedy}.
\end{equation*}
(To make these operators well-defined, we always select the same
$v$-greedy policy when applying each to any given $v$.)

Let $(V, \TT)$ be a regular ADP and suppose that the fundamental optimality
properties hold. Let $V_U$ be all $v \in V$ with $v \preceq Tv$.  Let $v^*$ denote
the value function. In this setting, it is said that 
\begin{itemize}
    \item \navy{value function iteration (VFI) converges} if $T^n v \uparrow
    v^*$ for all $v \in V_U$,
    \item \navy{optimistic policy iteration (OPI) converges} if $W_m^n v \uparrow
    v^*$ for all $v \in V_U$, and
    \item \navy{Howard policy iteration (HPI) converges} if $H^n v \uparrow
    v^*$ for all $v \in V_U$. 
\end{itemize}
Convergence of OPI implies convergence of VFI, since OPI reduces to VFI
when $m=1$.

\section{Fixed Point Results}\label{s:fprs}

The optimality and algorithmic results in \cite{sargent2025partially} rely on
well-posedness.  However, well-posedness is not trivial to establish.
Aiming for more directly applicable results, this section provides
several fixed-point results in ordered vector space that can be used to
establish well-posedness.    One result is drawn from \cite{marinacci2019unique}, 
and the other two are contraction-related arguments.  
Throughout this section, \emph{$E$ is a countably Dedekind complete Riesz
space}.

\subsection{Concavity and Fixed Points}

We will make use of a fixed point theorem due to \cite{marinacci2019unique}. To
begin, we fix $b \in E$ with $0 < b$. The \navy{lower perimeter} of the order
interval $V_b \coloneq [0, b]$ is defined as 
\begin{equation*}
    \partial V_b \coloneq 
    \setntn{v \in V_b}
    {\alpha \in \RR_+ \text{ and } \alpha b \leq v \implies \alpha = 0}.
\end{equation*}
(This follows the definition in \cite{marinacci2019unique} but is specialized to
the case where the set has the form $[0,b]$.) 

\begin{example}
    Let $E = bm\Xsf$, the space of bounded real-valued Borel measurable
    functions on $(\Xsf, \aA)$ with the pointwise partial order $\leq$.  Let $b$
    be any function in $E$ with $\inf_{x \in \Xsf} b(x) > 0$ and let $V_b = [0,
    b]$.  In this case, $\partial V_b$ is all $f \in V_b$ with $\inf_{x \in \Xsf}
    f(x) = 0$ (see Proposition~4 of \cite{marinacci2019unique}).
\end{example}

\begin{theorem}[Marinacci and Montrucchio]\label{t:mm}
    Let $S$ be an order continuous self-map on $V_b$.
    If $S$ is concave and $Sv \neq v$ whenever $v \in \partial V_b$, then $S$ has 
    exactly one fixed point
    $\bar v$ in $V_b$ and $S^n v \uparrow \bar v$  for any $v \in V_b$  with $v \leq S v$.
\end{theorem}

\begin{proof}
    This follows directly from  Theorem~1 of \cite{marinacci2019unique}.
\end{proof}

The existence component of Theorem~\ref{t:mm} is an immediate consequence of the 
Tarski--Kantorovich fixed-point theorem 
(see \cite{kantorovitch1939method}),
which says that an order-continuous self-map $S$ on a countably chain complete
poset $V$ has a least fixed point,  and also that $\bigvee_n S^n v$ belongs to the set of
fixed points for any $v \in V$ with $v \preceq S v$.
To see uniqueness, let $\bar v$ be the least element of the set of fixed
points, so that $\bar v \leq w$ for any fixed point $w$ of $S$.  Suppose that there is another fixed
point $\bar w$ with $\bar v < \bar w$. Let $v_n \coloneq \bar v - \frac{1}{n} 
(\bar w - \bar v$). Clearly, $v_n \uparrow \bar v$. Since $\bar v \notin \partial V_b$,
there exists some positive $\alpha \in \RR$ such that $\alpha b - \frac{2}{n} b \leq v_n$. 
Choosing some large enough $m \in \NN$, we have $0 < v_m$.  Note that
\begin{equation*}
    \bar v = \frac{m}{m+1} v_m + \frac{1}{m+1} \bar w.
\end{equation*}
By concavity of $S$,
\begin{equation*}
    \frac{m}{m+1} S v_m + \frac{1}{m+1} S \bar w
    \leq \ S \bar v
    = \bar v
    = \frac{m}{m+1} v_m + \frac{1}{m+1} \bar w.
\end{equation*}
Hence, $S v_m \leq v_m$ and so $S$ is a self-map on $[0, v_m]$. Applying the 
Tarski-Kantorovich fixed point theorem yields $\bar v \in [0, v_m]$. 
So we have $\bar w \leq \bar v$, which is a contradiction to $\bar v < \bar w$.

\subsection{Asymptotic Contractions and Fixed Points}

Another way to obtain unique fixed points is via contractivity conditions.
To discuss this issue, we begin with two definitions. Let $S$ be a self-map on
$V \subset E$.  We call $S$
\begin{itemize}
    \item \navy{asymptotically order contracting} on $V$ if $|S^n \, v - S^n w| \too 0$
    for any $v, w \in V$, 
        and
    \item \navy{absolutely order contracting} on $V$ if there exists an order
        continuous linear operator $K \colon E \to E$ such that, for all $v, w \in V$,
        \begin{equation}\label{eq:oc}
            |S \,v - S \, w| \leq K |v - w|
            \quad \text{and} \quad K^n | v - w | \too 0.
        \end{equation}
\end{itemize}

The asymptotic order contraction property is analogous to asymptotic
(or eventual) contraction in metric spaces, which is a weaker condition than
the contraction property. It is widely used in macroeconomics (see, e.g., \cite{kamihigashi2012order,
kamihigashi2014stochastic, stachurski2021dynamic, toda2024unbounded}).
The lemma below notes a useful connection between absolute order
contractions and asymptotic order contractions.

\begin{lemma}\label{l:ocfl}
    Let $S$ be an order preserving self-map on a subset $V$ of $E$.  
    If $S$ is absolutely order contracting, then $S$ is both order continuous
    and asymptotically order contracting.
\end{lemma}

\begin{proof}
    Let $E$, $V$ and $S$ be as stated. Regarding order continuity, fix
    $(v_n) \subset V$ with $v_n \uparrow v \in V$. We claim that $S \, v_n
    \uparrow S \, v$. We have $0 \leq S \, v - S \, v_n \leq K |v - v_n| = K (v
    - v_n) \downarrow 0$, where $\downarrow 0$ is by order continuity
    of $K$.  Hence $S \, v - S \, v_n \downarrow
    0$ and, therefore, $S \, v_n \uparrow S \, v$.

    Regarding the asymptotic order contraction property,
    note that the operator $K$ in \eqref{eq:oc} is order continuous and hence
    order preserving.  As a result, for any $v, w \in V$, we can iterate on the
    inequality in \eqref{eq:oc} to obtain $|S^n v - S^n w| \leq K^n |v - w|$ for
    all $n \in\NN$. By
    assumption, $K^n |v - w| \too 0$.  This proves the claim.
\end{proof}

Here is a preliminary fixed point result for an operator
    acting on a countably chain complete set.  In the applications below, we construct countably
chain complete sets by using order intervals. (Note that, in Riesz spaces, every
chain complete and order convex subset (A subset $A$ is order convex if $a \leq
c \leq b$ and $a, b \in A$ implies $c \in A$) is an order interval
\citep{marinacci2019unique}.)

\begin{lemma}\label{l:rieszc}
    Let $V$ be a countably chain complete subset of $E$ and let $S$ be an order
    continuous self-map on $V$.  If $S$ is asymptotically order contracting on
    $V$, then $S$ has a unique fixed point $\bar v \in V$ and $S^n v \too \bar
    v$ for any $v \in V$.
\end{lemma}

\begin{proof}
    Let $V$ and $S$ be as stated.  Since $S$ is an order continuous self-map on the 
    countably chain complete set $V$, the Tarski-Kantorovich fixed-point theorem
    implies that $S$ has at least one fixed point $\bar v \in V$.    
    Regarding convergence,
    $|S^n v - S^n \bar v| \too 0$ implies $S^n v \too \bar v$ for all $v \in V$.  
    Since order limits are unique (see, e.g., \cite{zaanen2012introduction}), this
    yields uniqueness of $\bar v$.
\end{proof}

\begin{remark}
    Let $V$ be chain complete and let $S$ be order preserving.
    In this case, the claim in Lemma~\ref{l:rieszc} still holds if we remove
    the order continuity condition.  This follows from 
    replacing the Tarski--Kantorovich fixed point theorem with the Knaster-Tarski
    fixed-point theorem, which tells us that every order-preserving self-map on
    a chain complete poset has a fixed point (see, e.g., \cite{davey2002introduction}).
\end{remark}

As a consequence of Lemma~\ref{l:rieszc}, absolutely order contracting maps obey
the following fixed-point result.

\begin{theorem}\label{t:rieszcon}
    Let $V$ be a countably chain complete subset of $E$ and let $S$ be
    an order preserving self-map on $V$. If $S$ is absolutely order 
    contracting on $V$, then $S$ has a unique fixed point $\bar v \in V$ and
    $S^n v \too \bar v$  for any $v \in V$. 
\end{theorem}

\begin{proof}
    Theorem~\ref{t:rieszcon} follows directly from Lemma~\ref{l:ocfl} and
    Lemma~\ref{l:rieszc}.
\end{proof}

\subsection{Affine Maps}

The next theorem is derived from Theorem~\ref{t:rieszcon} and applies to affine maps.

\begin{theorem}\label{t:nceccon}
    Let $S$ be a self-map on $E$.  If 
    \begin{enumerate}
        \item there exists a $h \in E$ and an order continuous linear operator $K$ on
            $E$ such that $Sv = h + Kv$ for all $v \in E$, and
        \item there is an $e \in E$ and $\rho \in [0, 1)$ with $|h| \leq e$ and
            $Ke \leq \rho e$,
    \end{enumerate}
    then $S$ has a unique fixed point $\bar v$ in the order interval $V \coloneq
    [-e/(1-\rho), e/(1-\rho)]$ and, moreover,
    \begin{equation*}
        S^n v \too \bar v 
        \quad \text{ for any }
        v \in V. 
    \end{equation*}

\end{theorem}

\begin{proof}
    Let the primitives be as stated, with conditions
    (i)--(ii) in force.  Note that, being order continuous, $K$ is
    order preserving and hence positive.  Observe that $S$ is a self-map on $V$, 
    since, for $v \in V$ we have
    \begin{equation*}
        |Sv| \leq |h| + |K v| 
        \leq |h| + K|v| 
        \leq e + K \frac{e}{1-\rho} 
        \leq e + \rho \frac{e}{1-\rho} 
        = \frac{e}{1-\rho}.
    \end{equation*}
    Hence $Sv \in V$.

    From (i) we obtain $|S v - Sw| \leq K|v - w|$ for any
    $v, w \in E$. Moreover, for any $v, w\in V$, we have
    \begin{equation*}
        K^n |v - w| 
        \leq K^n (|v|+|w|)
        \leq K^n \left(\frac{e}{1-\rho} + \frac{e}{1-\rho}\right)
        \leq K^n \frac{2 e}{1-\rho}
        \leq \rho^n \frac{2e}{1-\rho}.
    \end{equation*}
    By the Archimedean property, which holds in every countably Dedekind complete
    Riesz space (see \cite{marinacci2019unique}),
    the infimum on the right hand side is zero.
    Hence $S$ is absolutely order contracting. Since $E$ is countably Dedekind
    complete, $V$ is countably chain complete. 
    The claims now follow from Theorem~\ref{t:rieszcon}.
\end{proof}

\subsection{Application: Data Valuation}\label{ss:dv}

We close Section~\ref{s:fprs} with an application of the fixed point result in
Theorem~\ref{t:nceccon}.  The application involves the method for valuing
firm-specific private data proposed in \cite{veldkamp2023valuing}. 
(We treat it as a fixed point problem rather than a dynamic
program because profit maximization can be solved
period-by-period, without any dynamic considerations.)
 We consider a
slightly more realistic version of the model where the discount rate over cash
flows changes over time.  In particular,
analogous to \cite{veldkamp2023valuing}, the value of the firm at time $t$ is given by 
\begin{equation}\label{eq:rvv}
    V_t = \max_{k_t, \ell_t} 
    \left\{
        a(s_t) f(k_t, \ell_t) - w \ell_t - r k_t
    \right\}
    + \EE_t \beta_{t+1} V_{t+1},
\end{equation}
where $s_t$ is the stock of data,  $k_t$
is capital stock, $\ell_t$ is labor input, $w$ is the wage rate, $r$ is the
current rental rate of capital, and $\beta_{t+1}$ is the prevailing discount factor over
the time interval $t$. The first term on the right is current profits with
output price normalized to unity.
We suppose that 
\begin{itemize}
    \item $(\beta_t)$ is driven by stochastic kernel $Q$ on measurable space $(\Bsf, \bB)$
    and 
    \item $(s_t)$ is driven by stochastic kernel $P$ on measurable space $(\Ssf, \sS)$.
\end{itemize}
The function $a$ in \eqref{eq:rvv} is an $\sS$-measurable map from $\Ssf$ to
$\RR_+$. In this Markov environment, $V_t$ has the form $V_t = v(b_t, s_t)$ for
suitable choice of $v$. This Markov setting allows us to convert \eqref{eq:rvv}
into the functional equation
\begin{equation}\label{eq:vbd}
    v(b, s) = \pi(s) + \int \int b' v(b', s') Q(b, \diff b') P(s, \diff s'),
\end{equation}
where
\begin{equation*}
    \pi(s) \coloneq
    \max_{k, \ell} 
    \left\{
        a(s) f(k, \ell) - w \ell - r k
    \right\}.
\end{equation*}
Letting $x = (b, s)$, which takes values in the product space $\Xsf \coloneq
\Bsf \times \Ssf$, we can rewrite \eqref{eq:vbd} as 
\begin{equation}\label{eq:vx}
    v = \pi + K v 
    \quad \text{ on } \Xsf,
\end{equation}
where, overloading notation, $\pi(x) \coloneq \pi(b, s) \coloneq \pi(s)$, and, in
addition, $K$ is the positive linear operator defined by
\begin{equation*}
    (K f)(x) 
    = (Kf)(b, s) 
    = \int \int b' f(b', s') Q(b, \diff b') P(s, \diff s').
\end{equation*}
We call $v \in \RR^\Xsf$ a \navy{solution to the data valuation problem} if $v$ solves
\eqref{eq:vx}. To obtain such a solution we impose the conditions below.
Parts (i) and (ii) impose some regularity and part (iii) is a drift condition that
ensures that discounted
expected profits do not grow so fast that they prevent finite solutions.

\begin{assumption}\label{a:dv}
    The following conditions hold:
    \begin{enumerate}
        \item $\Bsf \subset [b_1, b_2] \subset (0, \infty)$,
        \item $\pi$ is in $bm\Xsf$ with $\pi_0 \coloneq \inf_{s \in S} \pi(s) >
        0$, 
        \item $\int \pi(s') P(s, \diff s') \leq \alpha \pi(s) + \lambda$ for some
            positive constants $\alpha, \lambda$ satisfying
            \begin{equation*}
                \alpha + \frac{\lambda}{\pi_0}
                < \frac{1}{b_2}.
            \end{equation*}
    \end{enumerate}
\end{assumption}

We now obtain a solution to the data valuation problem under the conditions in 
Assumption~\ref{a:dv}.
We begin with some preliminary facts.

\begin{lemma}\label{l:bmx}
    Let $V \subset \RR^{\Xsf}$ be closed under pointwise suprema. If $(v_n)$ is
    a sequence in $V$ and $v \in V$, then $v_n(x) \uparrow v(x) \text{ in } \RR 
    \text{ for all } x \in \Xsf$ if
        and only if $v_n \uparrow v$.
\end{lemma}

\begin{proof}
    ($\Rightarrow$) Suppose that $v_n(x) \uparrow v(x)$ in $\RR$ for all $x \in 
    \Xsf$.  Since $V$ is closed under pointwise suprema, we have $v \in V$, so
    $v$ is an upper bound of $(v_n)$.  For any other upper bound $w$,
    we have $v_n \leq w$ for all $n$, so $v_n(x) \leq w(x)$ for all $x \in \Xsf$.
    Hence, $v \leq w$.  This proves that $v = \bigvee_n v_n$.  Clearly, $(v_n)$
    is increasing, so $v_n \uparrow v$.  ($\Leftarrow$) Suppose that $v_n 
    \uparrow v$ for some $v \in V$.  Fix $x \in \Xsf$ and note that $v_n(x)$ is 
    increasing and bounded above by $v(x)$.  Hence the pointwise supremum function 
    $s(x) = \sup_n v_n(x)$ exists in $\RR^{\Xsf}$ and $s \leq v$.  
    Since $V$ is closed under pointwise suprema, we also have $s \in V$.  
    Since $v_n \leq s \leq v$ for all $n$ and 
    $\bigvee_n v_n = v$, we see that $s = v$.  Hence $v_n(x) \uparrow v(x)$.
\end{proof}

\begin{lemma}\label{l:k_oc}
    The positive linear operator $K$ is order continuous on $bm\Xsf$.
\end{lemma}

\begin{proof}
    Fix $v \in bm\Xsf$ and $v_n \uparrow
    v$ in $bm\Xsf$.  By Lemma~\ref{l:bmx}, we have $bv_n(b, s) \uparrow bv(b,
    s)$ for all $(b, s) \in \Xsf$. From the
    dominated convergence theorem we obtain $(Kv_n)(b, s) \uparrow
    (Kv)(b, s)$.  Applying Lemma~\ref{l:bmx} again yields $Kv_n \uparrow Kv$. 
\end{proof}

Let $e(b, s) \coloneq \frac{b_2}{b} \pi(s)$.  Clearly, $e \in bm\Xsf$ and $\pi \leq e$.

\begin{lemma}\label{l:rho}
    If Assumption~\ref{a:dv} holds, then there exists a $\rho \in [0, 1)$ such that 
    $Ke \leq \rho e$.
\end{lemma}

\begin{proof}
    For each $(b, s) \in \Xsf$, 
    \begin{align*}
        (Ke)(b, s) & = \int \int b' e(b', s') Q(b, \diff b') P(s, \diff s') \\
        & = \int \int b_2 \pi(s') Q(b, \diff b') P(s, \diff s') \\
        & = b_2 \int \pi(s') P(s, \diff s') \\
        & \leq b_2 (\alpha \pi(s) + \lambda) \\
        & = b_2 \left(\alpha + \frac{\lambda}{\pi(s)}\right) \pi(s) 
        \leq b_2 \left(\alpha + \frac{\lambda}{\pi_0}\right) \pi(s) 
        \leq b_2 (\alpha + \frac{\lambda}{\pi_0}) e(b, s).
    \end{align*}
    With $\rho \coloneq b_2 (\alpha + \frac{\lambda}{\pi_0}) < 1$ 
    we have $Ke \leq \rho e$.
\end{proof}

In the next proposition, we set $Sv \coloneq \pi + Kv$, while $e, \rho$ are as
stated in Lemma~\ref{l:rho}.

\begin{proposition}\label{p:vel}
    If Assumption~\ref{a:dv} holds, then
    \begin{enumerate}
        \item there exists a $v^* \in bm\Xsf$ that solves the data valuation problem, and
        \item $S^n \; v \too v^*$ for any $v \in [-e/(1-\rho), e/(1-\rho)] \subset bm\Xsf$.
    \end{enumerate}
\end{proposition}

\begin{proof}
    The fixed point of $S$ satisfies \eqref{eq:vx}.  By Lemma~\ref{l:k_oc}, $K$ is order
    continuous.  We showed above Lemma~\ref{l:rho} that there is some $e \in bm\Xsf$ 
    with $\pi \leq e$. Also, in that lemma, we showed that there is some $\rho \in [0, 1)$
    such that $K e \leq \rho e$.  
    All claims now follow from Theorem~\ref{t:nceccon}.
\end{proof}

Proposition~\ref{p:vel} offers a constructive solution to the data
valuation problem from a given set of parameters. For empirical
studies, the next step is to estimate or calibrate these structural
parameters using aggregate data. In practice, estimating these parameters
typically requires numerical solution of \eqref{eq:vbd},
and the numerical procedure can exploit the iterative algorithm
described in part (ii) of the proposition.

\section{Optimality}\label{s:opti}

This section uses the fixed point theorems in Section~\ref{s:fprs} to extend the
optimality results from \cite{sargent2025partially}. Throughout this section, 
we say that $(V, \TT)$ is 
\begin{itemize}
    \item \navy{order continuous} if each $T_\sigma \in \TT$ is order continuous in $V$,
    \item \navy{asymptotically order contracting} if each $T_\sigma \in \TT$ is
        asymptotically order contracting on $V$, and
    \item \navy{absolutely order contracting} if each $T_\sigma \in \TT$ is
        absolutely order contracting on $V$.
\end{itemize}

In all of Section~\ref{s:opti},
\begin{itemize}
    \item $(V, \TT)$ is a regular ADP and
    \item $V$ is a subset of a countably Dedekind complete Riesz space $E = (E, \leq)$.
\end{itemize}

\subsection{Concave ADPs}\label{ss:concave}

This section applies the fixed-point result for concave operators from
Theorem~\ref{t:mm} (originally proved in \cite{marinacci2019unique}), to
construct the optimality results for ADPs where policy operators are
concave. In the result below, $V = [0, b]$ for some $0 < b \in E$.  

\begin{theorem}\label{t:riesz_con}
    If $(V, \TT)$ is order continuous, every $T_\sigma \in
    \TT$ is concave on $V$, and $T_\sigma \, v \neq v$ whenever $v
    \in \partial V$, then 
    \begin{enumerate}
        \item the fundamental optimality properties hold, and
        \item VFI, OPI and HPI all converge.
    \end{enumerate}
\end{theorem}

\begin{proof}[Proof of Theorem~\ref{t:riesz_con}]
    Let $(V, \TT)$ be as stated. Since $E$ is countably Dedekind complete,
    $V$ is countably chain complete. As $(V, \TT)$ is order continuous, 
    Theorem~5.3 of \cite{sargent2025partially} implies that the conclusions of
    Theorem~\ref{t:riesz_con} hold whenever $(V, \TT)$ is 
    well-posed.  This well-posedness follows directly from the conditions in
    Theorem~\ref{t:riesz_con} and the fixed point result in Theorem~\ref{t:mm}.
\end{proof}

\subsection{Order Contracting ADPs}

This section applies the fixed-point results in Lemma~\ref{l:rieszc} and 
Theorem~\ref{t:rieszcon} to build optimality results for ADPs
where policy operators have long run contractivity properties. 
We show below that these long-run contractivity conditions encompass dynamic
programs with time-varying discount rates.

\begin{theorem}\label{t:rieszcadp}
    Let $V$ be countably chain complete. If $(V, \TT)$ is order continuous
    and asymptotically order contracting, then
    \begin{enumerate}
        \item the fundamental optimality properties hold, and
        \item VFI, OPI and HPI all converge.
    \end{enumerate}
\end{theorem}

\begin{proof}[Proof of Theorem~\ref{t:rieszcadp}]
    Let the stated conditions hold.  In view of Theorem~5.3 of
    \cite{sargent2025partially}, we only need to show that $(V, \TT)$ is
    well-posed.  Well-posedness follows from Lemma~\ref{l:rieszc}. 
\end{proof}

\begin{remark}\label{rmk:cc}
    If $V$ is also chain complete, then the claims in Theorem~\ref{t:rieszcadp} 
    hold without order continuity.
    Since the Bellman operator $T$ is an order preserving
    self-map on the chain complete set $V$, the Knaster-Tarski
    theorem implies that $T$ has at least one fixed point in $V$.  
    The first claim then follows from Theorem~5.1 of \cite{sargent2025partially}.  
    (The required downward stability is a consequence of Lemma~\ref{l:rieszc}.) 
    Regarding convergence, let $v_\perp$ be the least element of $V$,
    let $v^*$ be the value function, and let $\sigma$ be an optimal
    policy. Fix $v \in V_U$.  We have $v_\perp \leq T_\sigma v_\perp \leq T v$ and hence, 
    iterating on this inequality, $T_\sigma^n v_\perp \leq T^n v$ for all $n \in \NN$. 
    By Lemma~\ref{l:rieszc}, we have $T_\sigma^n v_\perp \uparrow v_\sigma = v^*$. By 
    Lemma~6.2 of \cite{sargent2025partially}, $T^n v \leq v^*$, hence, 
    $T^n v \uparrow v^*$.  Convergence of OPI and HPI follows from Corollary~6.3 of
    \cite{sargent2025partially}.
\end{remark}

The next result is a direct consequence of Theorem~\ref{t:rieszcadp}.

\begin{theorem}\label{t:rieszbk}
    Let $V$ be countably chain complete. If $(V, \TT)$ is absolutely order contracting,
    then 
    \begin{enumerate}
        \item the fundamental optimality properties hold, and
        \item VFI, OPI and HPI all converge.  
    \end{enumerate}
\end{theorem}

\begin{proof}
    Since $(V, \TT)$ is an ADP, each $T_\sigma$ is, by
    definition, order preserving. 
    Under the stated conditions, Lemma~\ref{l:ocfl} implies that 
    $(V, \TT)$ is order continuous and asymptotically order contracting.
    The claims now follow from Theorem~\ref{t:rieszcadp}.
\end{proof}

The final result of this section applies the fixed-point result in
Theorem~\ref{t:nceccon} to build an optimality result for ADP where policy
operators are affine mappings.  Let $(E, \TT)$ be an ADP. We call $(E, \TT)$
\navy{affine} if each $T_\sigma \in \TT$ has the form
\begin{equation*}
    T_\sigma \, v = r_\sigma + K_\sigma \, v
    \qquad (v \in E)
\end{equation*}
for some $r_\sigma \in E$ and order continuous linear operator 
$K_\sigma \colon E \to E$.
In this setting, we say that \navy{Condition~C holds} if 
\begin{equation*}
    \text{$\exists \; e \in E$ and $\rho \in [0, 1)$ with $|r_\sigma| \leq e$ and
        $K_\sigma \, e \leq \rho e$ for all $\sigma \in \Sigma$}.    
\end{equation*}
When Condition~C holds, we let
\begin{equation*}
    V \coloneq [a, b] 
    \quad \text{where } 
    a \coloneq - \frac{e}{1-\rho}
    \text{ and }
    b \coloneq \frac{e}{1-\rho}.
\end{equation*}

\begin{theorem}\label{t:rieszaff}
    Let $(E, \TT)$ be a regular ADP.  
    If $(E, \TT)$ is an affine ADP and Condition~C holds, then
    \begin{enumerate}
        \item $(V, \TT)$ is also an ADP,
        \item the fundamental optimality properties hold for $(V, \TT)$, and
        \item VFI, OPI and HPI all converge in $(V, \TT)$.  
    \end{enumerate}
\end{theorem}

\begin{proof}
    Let the stated conditions hold. Fix $T_\sigma \in \TT$, by order continuity,
    $K_\sigma$ is order preserving and positive. For any $v \in V$, we have
    \begin{equation}
        |T_\sigma v| \leq |r_\sigma| + |K_\sigma v| 
        \leq |r_\sigma| + K_\sigma|v| 
        \leq e + K_\sigma \frac{e}{1-\rho} 
        \leq e + \rho \frac{e}{1-\rho} 
        = \frac{e}{1-\rho}.
    \end{equation}
    This shows that $T_\sigma \, v \in V$ and, therefore, $(V, \TT)$ is an ADP. 
    Since $E$ is
    countably Dedekind complete, the order interval $V$ is countably chain complete.  
    Theorem~\ref{t:nceccon} implies that $(V, \TT)$ is well-posed.  
    Since $(E,\TT)$ is regular, $(V, \TT)$ is also regular. 
    Since each policy operator is order
    continuous on $E$, the ADP $(V, \TT)$ is order continuous.
    All claims now follow from Theorem~5.3 of \cite{sargent2025partially}.
\end{proof}

\begin{figure}[H]
    \centering
    \begin{tikzpicture}[
        node distance = 1.5cm and 1.2cm,
        box/.style = {rectangle, draw, rounded corners, align=center, 
                      minimum width=3.4cm, minimum height=1cm},
        arrow/.style = {->, thick},
        outerbox/.style = {rectangle, draw, rounded corners, dashed, inner sep=0.5cm}
    ]

    % Nodes
    \node[box] (concavity) {concavity\\no fixed point in lower perimeter};
    \node[box, below=of concavity] (asymptotic) {asymptotically order contracting};
    \node[box, below=of asymptotic] (absolute) {absolutely order contracting};

    \node[box, right=of asymptotic] (well) {well-posedness};
    \node[box, right=of well] (fundamental) {fundamental\\optimality results};
    \node[box, below=of fundamental] (regularity) {regularity};

    % Arrows
    \draw[arrow] (concavity) -- (well);
    \draw[arrow] (asymptotic) -- (well);
    \draw[arrow] (absolute) -- (asymptotic);
    \draw[arrow] (well) -- (fundamental);
    \draw[arrow] (regularity) -- (fundamental);

    % Outer box
    \node[outerbox, fit=(concavity) (absolute) (fundamental) (regularity)] (all) {};
    \end{tikzpicture}
    \caption{Directed graph showing implications
    among the major theoretical concepts}
    \label{fig:concept}
\end{figure}

Figure~\ref{fig:concept} shows a directed graph listing
implications between the major concepts discussed in the theoretical
component of the paper.

\section{ADP Applications}\label{s:appl}

Next we consider applications of the optimality results stated
above.  Throughout this section, we consider an agent who interacts with
a state process $(Y_t, Z_t)_{t \geq 0}$ by choosing an action path $(A_t)_{t \geq 0}$
to maximize lifetime utility in a Markov setting, where $Y_t$ is endogenous 
and $Z_t$ is exogenous.  Let $\Xsf \coloneq \Ysf \times \Zsf$ be a state space
where $(Y_t, Z_t)$ takes values
and let $\Asf$ be an action space where $A_t$ takes values.  
Assume that $\Asf$ is a separable metrizable space.
The agent responds to the state $x \coloneq (y, z)$ by choosing action $a$ from 
$\Gamma(x) \subseteq \Asf$, 
where $\Gamma$ is a nonempty correspondence from $\Xsf$ to $\Asf$.  
Let 
\begin{equation*}
    \Gsf = \left\{
        (x, a) \in \Xsf \times \Asf \,:\, a \in \Gamma(x)
    \right\}.
\end{equation*}
Let $r \colon \Gsf \to \RR$ be a reward function and let $P$ be a stochastic kernel
on $\Zsf$.  Given a feasible state-action pair $(x, a)$, reward $r(x, a)$ is received,
the next period exogenous state $z'$ is randomly drawn from $P(z, \cdot)$ and the next
period endogenous state $y'$ is governed by 
\begin{equation*}
    y' = H(y, a, z'),
\end{equation*}
where $H$ is measurable and maps $\Ysf \times \Asf \times \Zsf$ to $\Ysf$. 
Future rewards are discounted via a stochastic discounting process 
$(\beta_t)_{t \geq 0}$.

\begin{assumption}\label{a:regular}
    The following conditions hold:
    \begin{enumerate}
        \item the correspondence $\Gamma \colon \Xsf \to \Asf$ is 
            compact-valued and continuous,
        \item the reward function $r$ is non-negative, bounded and continuous on $\Gsf$,
            and
        \item the map $a \mapsto \int v(H(y, a, z'), z') P(z, \diff z')$
            is continuous on $A$ for any $v \in bm\Xsf$ and $(y, z) \in \Xsf$.
    \end{enumerate}
\end{assumption}

Assumption~\ref{a:regular} is used to obtain regularity of the associated ADP,
as will be discussed below.  We focus on the regular case to constrain
complexity. While the conditions in Assumption~\ref{a:regular} are
somewhat strict, they do hold in many important cases.

\begin{example}
    If action space $A$ is finite, then condition (iii) of Assumption~\ref{a:regular}
    always holds.  (Condition (i) can be removed in this setting.) 
\end{example}

\begin{example}\label{exm:iid}
    Let $\Ysf \subseteq \RR^n$ and $\Zsf \subseteq \RR^n$. 
    Suppose the exogenous shock $Z_t$ is {\sc iid} with a continuous density $\psi$.  
    We reduce the state space from $\Xsf = \Ysf \times \Zsf$ to $\Ysf$. 
    Assume that
    \begin{enumerate}
        \item for a given $(y, a) \in \Ysf \times \Asf$, the map
        $z' \mapsto y' = H(y, a, z')$ is a $C^1$-diffeomorphism from
        $\Zsf$ to $\Ysf$, and
        \item the map $a \mapsto z' = H^{-1}(y, a, y')$ and 
        $a \mapsto |J(H^{-1})(y, a, y')|$
        are continuous for any given $(y, y') \in \Ysf \times \Ysf$,
        where 
        $J(H^{-1})(y, a, y') \coloneq 
        \det \left( \frac{\diff H^{-1}(y, a, y')}{\diff y'} \right)$
        is the Jacobian determinant of the transformation $H^{-1}(y, a)$.
    \end{enumerate}
    Then the condition (iii) of Assumption~\ref{a:regular} holds.
    To see this, fix $v \in bm\Xsf$ and $M \in \RR_+$ such that
    $|v| \leq M$. For given $y \in \Ysf$, let the map 
    $f \colon \Asf \to \RR$ be defined by
    $f(a) = \EE[v \circ H(y, a, Z)]$.  We need to show that $f$ is
    continuous on $A$.  To this end, we introduce $Y \coloneq H(y, a, Z)$  
    and observe that the density $h$ of $Y$ is 
    \begin{equation*}
        h(y') = \psi(H^{-1}(y, a, y')) |J(H^{-1})(y, a, y')|,
    \end{equation*}
    as follows from the multivariate change of variables formula 
    (see, e.g., \cite{schwartz1954}). For any $B \subseteq \Ysf$,
    \begin{equation*}
        P[Y \in B] 
        = P[H(y, a, Z) \in B] 
        = P[Z \in H^{-1}(y, a, B)] = \int_{H^{-1}(y, a, B)} \psi(z') \diff z'.
    \end{equation*}
    Using the change of variables $z' = H^{-1}(y, a, y')$, 
    $\diff z' = |J(H^{-1})(y, a, y')| \diff y'$,
    we have
    \begin{equation*}
        P[Y \in B] 
        = \int_B \psi(H^{-1}(y, a, y')) |J(H^{-1})(y, a, y')| \diff y'.
    \end{equation*}
    So it follows that the stated $h$ is a density for $Y = H(y, a, Z)$.
    Now we can rewrite the map $f$ as
    \begin{equation*}
        f(a) = \int v(y') \psi(H^{-1}(y, a, y')) |J(H^{-1})(y, a, y')| \diff y'.
    \end{equation*}
    Let $g(a, y') \coloneq \psi(H^{-1}(y, a, y')) |J(H^{-1})(y, a, y')|$, 
    the map $a \mapsto g(a, y')$ is continuous by assumption. 
    So for any $a_n \to a \in \Asf$, $g(a_n, y') \to g(a, y')$.
    Note that $g(a_n, y')$ and $g(a, y')$ are the densities for $Y = H(y, a_n, Z)$
    and $Y = H(y, a, Z)$, so $\int g(a_n, y') \diff y' = \int g(a, y') \diff y' = 1$.
    By Scheffe's lemma,
    \begin{equation*}
        \int |g(a_n, y') - g(a, y')| \diff y' \to 0.
    \end{equation*}
    Note that 
    \begin{align*}
        |f(a_n) - f(a)|
        & \leq \int |v(y')| \cdot |g(a_n, y') - g(a, y')| \diff y' \\
        & \leq M \int |g(a_n, y') - g(a, y')| \diff y'.
    \end{align*}
    Hence, $f(a_n) \to f(a)$.
\end{example}

\subsection{Risk-Sensitive Preferences}\label{ss:recpref}

This section provides an application of the concave ADP result in
Theorem~\ref{t:riesz_con}, which is based on the fixed point result from
\cite{marinacci2019unique}.  In the application we suppose that lifetime utility
$V_t$ is governed by risk-sensitive preferences \citep{bauerle_markov_2022,
bauerle_stochastic_2018, bauerle_markov_2024}, so that lifetime value is given
recursively by 
\begin{equation*}
    V_t = R_t + \frac{\beta_t}{\theta} \ln \left\{
        E_t e^{(\theta V_{t+1})}
    \right\}.
\end{equation*}
Here $(R_t)$ is a stochastic reward process and $\theta$ is a nonzero constant in $\RR$.  
Decreasing $\theta$ lowers appetite for risk, 
and increasing $\theta$ does the opposite. 
Here we focus on the standard case $\theta < 0$.  
Future rewards are discounted via a state-dependent discount factor 
$\beta \colon \Zsf \to (0,\infty)$. (For a discussion of the value and significance of
state-dependent discounting, see, e.g., \cite{stachurski2021dynamic}.)  
The Bellman equation corresponding to this Markov decision process is
\begin{equation}\label{eq:bers_infinite}
    v(x) \coloneq 
    \max_{a \in \Gamma(x)} \left\{
        r(x, a) + \frac{\beta(z)}{\theta} \ln \left [
        \int e^{\theta v(H(y, a, z'), z')} P(z, \diff z')
    \right ]
    \right\} \text{ for all } x \in \Xsf,
\end{equation}
where $y$ and $z$ are determined by $x = (y, z)$.

We impose some stability conditions.

\begin{assumption}\label{a:concave}
    The following conditions hold:
    \begin{enumerate}
        \item the reward function $r$ satisfies
            $0 < \underline{r} \leq r \leq \bar r$,
        \item the map $\bar H(y, z') \coloneq \sup_{a \in \Asf} H(y, a, z')$ 
            mapping from $\Xsf$ to $ \Ysf$ is well-defined and measurable, and 
            for any given $z' \in \Zsf$, it is order-preserving on $\Ysf$,
        \item the discount factor $\beta$ is bounded and Borel measurable, and
        \item there exists a $n \in \NN$ such that
            \begin{equation*}
                \sup_{z \in \Zsf} \EE^{P}_z \prod_{t=0}^{n-1} \beta(Z_t) < 1.
            \end{equation*}
    \end{enumerate}
\end{assumption}

In (iv), $\EE^{P}$ is conditional expectation with respect to the stochastic
kernel $P$.

To solve \eqref{eq:bers_infinite}, we construct an ADP in the Riesz space $bm\Xsf$, 
the bounded Borel measurable functions on $\Xsf$ with the pointwise partial 
order $\leq$. Let $S$ be an operator defined at 
$v \in (bm\Xsf)_+$ by
\begin{equation*}
    (Sv)(x) = \bar r + \beta(z) \int v(\bar H(y, z'), z') P(z, \diff z')
    \text{ for all } x = (y, z) \in \Xsf.
\end{equation*}
By (iv) of Assumption~\ref{a:concave} and Theorem~2.1 of \cite{stachurski2021dynamic}, 
$S$ is eventually contracting and thus globally stable on $(bm\Xsf)_+$.  
Let $b \in (bm\Xsf)_+$ be the unique fixed point of $S$. 
Clearly, $0 < \bar r \leq b$, and 
by (ii) of Assumption~\ref{a:concave} and Theorem~2.2 of \cite{stachurski2021dynamic}, 
$b$ is increasing in $\Ysf$.  
Let $V_b \coloneq [0, b] \subset (bm\Xsf)_+$ be the value space. 
A feasible policy is a Borel measurable map $\sigma \colon \Xsf \to \Asf$
such that $\sigma(x) \in \Gamma(x)$ for all $x \in \Xsf$.
Let $\Sigma$ be the set of all feasible policies.
To each policy $\sigma \in \Sigma$, we assign a policy operator $T_\sigma$ 
defined at $v \in V_b$ by
\begin{equation*}
    (T_\sigma v)(x) =
    r(x, \sigma(x)) + \frac{\beta(z)}{\theta} \ln \left [
        \int e^{\theta v(H(y, \sigma(x), z'), z')} P(z, \diff z')
    \right ].
\end{equation*}
Let $\TT_{RS} = \left\{ T_\sigma \, : \, \sigma \in \Sigma \right\}$.

\begin{lemma}\label{l:adp}
    If Assumption~\ref{a:concave} holds, then $(V_b, \TT_{RS})$ is an ADP.
\end{lemma}

\begin{proof}
    To see that $T_\sigma$ is order preserving, fix $\sigma \in \Sigma$, $x
    \in \Xsf$ and $v \leq w$ in $V_b$.  Since $\theta < 0$ we have $e^{\theta v}
    \geq e^{\theta w}$ pointwise on $\Xsf$, so, for any $x = (y, z) \in \Xsf$,
    letting $x' = (H(y, \sigma(x), z'), z')$, we have
    \begin{equation*}
        \ln \left [
        \int e^{\theta v(x')} P(z, \diff z')
            \right ] 
        \geq
        \ln \left [
        \int e^{\theta w(x')} P(z, \diff z')
            \right ].
    \end{equation*}
    \begin{equation*}
        \fore
        \frac{\beta(z)}{\theta} \ln \left [
        \int e^{\theta v(x')} P(z, \diff z')
            \right ] 
            \leq
        \frac{\beta(z)}{\theta} \ln \left [
        \int e^{\theta w(x')} P(z, \diff z')
            \right ].
    \end{equation*}
    \begin{equation*}
       \fore T_\sigma \, v \leq T_\sigma \, w. 
    \end{equation*}
    Regarding $T_\sigma V_b \subseteq V_b$, fix $\sigma \in \Sigma$ and $v \in V_b$.  
    The arithmetic and integral operations in the definition of $T_\sigma$ all
    preserve measurability, so $T_\sigma \, v$ is Borel measurable.  In
    addition, $T_\sigma \, v$ is bounded and nonnegative when $v \in V_b$, because
    for such a $v$ there is a constant $m$ such that $0  \leq v \leq m$ on
    $\Xsf$.  Hence, for any $x \in \Xsf$,
    \begin{equation*}
        0 \leq r(x, \sigma(x)) 
        \leq (T_\sigma v)(x)
        \leq r(x, \sigma(x)) + \beta(x) m
        \leq \sup r + m \sup \beta .
    \end{equation*}
    To see $T_\sigma v \leq b$, 
    let an order preserving operator $S_\sigma$ defined at each $v \in V_b$ by
    \begin{equation*}
        (S_\sigma v)(x) 
        = r(x, \sigma(x)) 
        + \beta(z) \int v(H(y, \sigma(x), z'), z') P(z, \diff z')
        \text{ for all } x \in \Xsf.
    \end{equation*}
    Consider the random variable $\eta \coloneq v(X')$,
    where $X' = (H(y, \sigma(x), Z'), Z')$ is a draw from $P(z, \cdot)$.
    Let $f(t) \coloneq \exp{\theta t}$.  The function $f$ is convex because 
    $f''(t) = \theta^2 \exp{\theta t} > 0$.  Hence, by Jensen's inequality, we
    have $\EE f(\eta) \geq
    f(\EE(\eta))$.  With $\theta < 0$, $f^{-1}(t) = \frac{1}{\theta} \ln{t}$ is
    decreasing, so $f^{-1} (\EE f(\eta)) \leq \EE (\eta)$.  Therefore,
    letting $x' = (H(y, \sigma(x), z'), z')$, we have
    \begin{align*}
        \frac{\beta(z)}{\theta} \ln \left\{ \int \exp{[\theta v(x')]} 
        P(z, \diff z') \right\}
        \leq \beta(z) \int v(x') P(z, \diff z').
    \end{align*}
    This yields $(T_\sigma v)(x) \leq (S_\sigma v)(x)$ 
    for all $v \in V_b$ and $x \in \Xsf$.
    Since $b$ is increasing in $\Ysf$, $S_\sigma b \leq S b$.
    Hence, $T_\sigma v \leq S_\sigma v \leq S_\sigma b \leq S b = b$.
\end{proof}

\begin{lemma}\label{l:regular}
    If Assumption~\ref{a:regular} holds, then $(V_b, \TT_{RS})$ is regular.
\end{lemma}

\begin{proof}
    Fix $v \in V_b$.  By (ii) and (iii) of Assumption~\ref{a:regular}, the map
    \begin{equation*}
        (x, a) \mapsto r(x, a) + \frac{\beta(z)}{\theta} \ln \left [
        \int e^{\theta v(H(y, a, z'), z')} P(z, \diff z')
    \right ]
    \end{equation*}
    is measurable on $\Xsf$ and continuous on $\Asf$, so, by (i) of
    Assumption~\ref{a:regular} and the measurable
    maximum theorem (see 
    Theorem~18.19 of \cite{aliprantis2006infinite}), there exists a
    measurable function $\sigma \colon \Xsf \to \Asf$ such that
    \begin{equation}
        \sigma(x) \in \argmax_{a \in \Gamma(x)} \left\{
        r(x, a) + \frac{\beta(z)}{\theta} \ln \left [
        \int e^{\theta v(H(y, a, z'), z')} P(z, \diff z')
        \right ]
        \right\} \text{ for all } x \in \Xsf.
    \end{equation}
    For this $\sigma$ and any other $\tau \in \Sigma$, we have $(T_\tau v)(x)
    \leq (T_\sigma v)(x)$ for all $x \in \Xsf$. In particular, $\sigma$ is $v$-greedy.
\end{proof}

\begin{lemma}\label{l:continuous}
    Each $T_\sigma \in \TT_{RS}$ is order continuous.
\end{lemma}

\begin{proof}
    Fix $\sigma \in \Sigma$ and $v_n \uparrow v \in V_b$. By
    Lemma~\ref{l:bmx}, we have $v_n(x') \uparrow v(x')$ for all
    $x' \in \Xsf$, so $e^{\theta v_n(x')} \downarrow e^{\theta v(x')}$.
    Fixing $x \in \Xsf$ and applying the dominated convergence theorem, we
    obtain
    \begin{equation*}
        \int e^{\theta v_n(x')} P(z, \diff z') \downarrow 
        \int e^{\theta v(x')} P(z, \diff z'),
    \end{equation*}
    where $x' = (H(y, \sigma(x), z'), z')$.
    Since the log function is continuous, this yields $(T_\sigma v_n)(x)
    \uparrow (T_\sigma v)(x)$.   Applying Lemma~\ref{l:bmx} again
    yields $T_\sigma v_n \uparrow T_\sigma v$, so $T_\sigma$ is order continuous.
\end{proof}

\begin{lemma}\label{l:concave}
    Every $T_\sigma \in \TT_{RS}$ is 
    concave on $V_b$.
\end{lemma}

\begin{proof}
    Fix $\sigma \in \Sigma$. $T_\sigma$ is concave since for all $v, w \in V$ and 
    $0 < \lambda < 1$, let $p = \frac{1}{\lambda} > 1$ and $q = \frac{1}{1-\lambda} > 1$,
    we have
    \begin{align*}
        & \lambda \frac{1}{\theta} \ln{\EE [\exp{(\theta v)}]} + (1 - \lambda) 
        \frac{1}{\theta}
        \ln{\EE [\exp{(\theta w)}]} \\
        =& \frac{1}{\theta} \ln{\left\{(\EE [\exp{(\theta v)}])^{\lambda} 
        (\EE [\exp{(\theta w)}])^{1- \lambda}\right\}} \\
        =& \frac{1}{\theta} \ln{\left\{(\EE [\exp{(p \lambda \theta v)}])^{\frac{1}{p}}
        (\EE [\exp{(q (1 - \lambda) \theta w)}])^{\frac{1}{q}}\right\}} \\
        \leq & \frac{1}{\theta} \ln{\left\{\EE[\exp{(\lambda \theta v)} \cdot 
        \exp{(1 - \lambda) \theta w)}]\right\}}\\
        =& \frac{1}{\theta} \ln{\left\{\EE[\exp{(\lambda \theta v + (1 - \lambda) 
        \theta w)}]\right\}}.
    \end{align*}
    The inequality uses the fact that the map $w \mapsto \frac{1}{\theta} \ln w$
    is decreasing, as well as Holder's inequality. Hence, for 
    any $x \in \Xsf$, we have
    \begin{equation*}
        (\lambda T_\sigma v)(x) + ((1 - \lambda) T_\sigma w)(x)
        \leq (T_\sigma (\lambda v + (1 - \lambda) w))(x).
        \qedhere
    \end{equation*}
\end{proof}

\begin{proposition}\label{p:rs}
    If Assumption~\ref{a:regular} and Assumption~\ref{a:concave} hold, 
    then the fundamental optimality properties hold for $(V_b, \TT_{RS})$, 
    and VFI, OPI and HPI all converge in $(V_b, \TT_{RS})$.
\end{proposition}

\begin{proof}
    We showed in Lemma~\ref{l:adp} that $(V_b, \TT_{RS})$ is an ADP.  
    In Lemma~\ref{l:regular} and Lemma~\ref{l:continuous}, 
    we showed that $(V_b, \TT_{RS})$ is regular and order continuous.
    We proved in Lemma~\ref{l:concave} that each policy operator is concave.
    Hence, in view of Theorem~\ref{t:riesz_con}, 
    all the claims hold for $(V_b, \TT_{RS})$ if 
    for each $T_\sigma \in \TT_{RS}$,  $T_\sigma v \neq v$ whenever $v \in \partial V_b$.
    Since $b \in V_b$ with $\inf_{x \in \Xsf} b(x) \geq \bar{r} > 0$, 
    we have $\partial V_b$ is all $f \in V$ 
    with $\inf_{x \in \Xsf} f(x) = 0$. 
    But by (i) of Assumption~\ref{a:concave}, for any $v \in V_b$, we have $0 <
    \underline{r} \leq r \leq T_\sigma v$,
    so $\inf_{x \in \Xsf} (T_\sigma v)(x) \geq \underline{r} > 0$. 
    Hence, $T_\sigma v \notin \partial V_b$, and thus $T_\sigma v \neq v$
    for any $v \in \partial V_b$.  
\end{proof}

\subsubsection{A Numerical Example: Firm Valuation with Optional Exit}

In recent years, there has been significant interest in the impact of 
risk aversion on behalf of the managers of firms, 
as opposed to the more standard assumption of risk neutrality. (See, for
example, \cite{graham2013managerial}, \cite{brenner2015risk}, or
\cite{koudstaal2016risk}.) Here we consider a risk-averse firm that can either
(i) choose to exit a market and receive a scrap value or (ii) remain in
operation, receive its current profit, and reassess next period.
Our aim is to apply the theoretical results derived above, while examining the
effect of risk-aversion on optimal management decisions.

The lifetime value of the firm in state $x$ is
\begin{equation*}
    v(x) = \max \left\{ 
    s, \pi(x) + \frac{\beta(x)}{\theta} \ln \left[
    \sum_{x'} e^{\theta v(x')} P(x, x')
    \right]
    \right\},
\end{equation*}
where $s$ represents a fixed scrap value and $\pi(x)$ represents current profit.  The 
state $x$ is exogenous productivity, which evolves on a finite set $\Xsf \subset
\RR_{++}$, while $P$ is a stochastic matrix.

We adopt the discount factor process used in \cite{hills2019effective}.  This
takes the form $\beta_t = \kappa X_t$, where $(X_t)_{t \geq 0}$ evolves
according to $X_{t+1} = 1 - \rho + \rho X_t + \alpha \epsilon_{t+1}$, with
$(\epsilon_t)$ IID and standard normal. Following \cite{hills2019effective}, we
discretize the dynamics of $(X_t)_{t \geq 0}$ via a Tauchen approximation
\citep{tauchen1986finite}, producing the stochastic matrix $P$ on a finite set
$\Xsf = \{x_1, x_2, \cdots, x_n\}$ discussed above. In the computation
described below, we set $\rho = 0.85, \alpha = 0.0062, \kappa = 0.99875, n =
400, s = 100$ and the profit function is $\pi(x) = x$. 

We solve the dynamic program using VFI, HPI, and OPI.
Figure~\ref{fig:vs} shows the output of VFI. In the
left subfigure we plot an approximation of $v^*$ with parameter $\theta = -1$, 
computed as a limit of VFI
iterations, as well as the exit option $s =
100$, and the function $h^*$ given by
$$
    h^*(x) = \pi(x) +
        \frac{\beta(x)}{\theta} 
        \ln \left[\sum_{x'} e^{\theta v^*(x')} P(x, x')\right],
$$
which is the value of remaining in operation.  The value of $x$ at the kink in
$v^*$ is the threshold productivity, or the minimum productivity level
at which the firm chooses to remain in operation.

In the right subfigure, we vary the risk attitude parameter $\theta$ from $-2$ to
$-0.1$ and plot threshold productivity as a function of $\theta$, holding
other parameters fixed.  The threshold productivity
decreases with $\theta$, as the firm becomes progressively less risk-averse.
This is because decreasing risk aversion means that gambles over future profits become more
attractive, relative to the constant scrap value,
which favors continuing over exiting.  Figure~\ref{fig:policy} plots the optimal
policy against the stationary distribution of the productivity
process (scaled up by a factor of 150 to aid visibility).  Parameter values are
the same as the left subfigure in Figure~\ref{fig:vs}.

\begin{figure}[h]
    \centering
    \includegraphics[width=\linewidth]{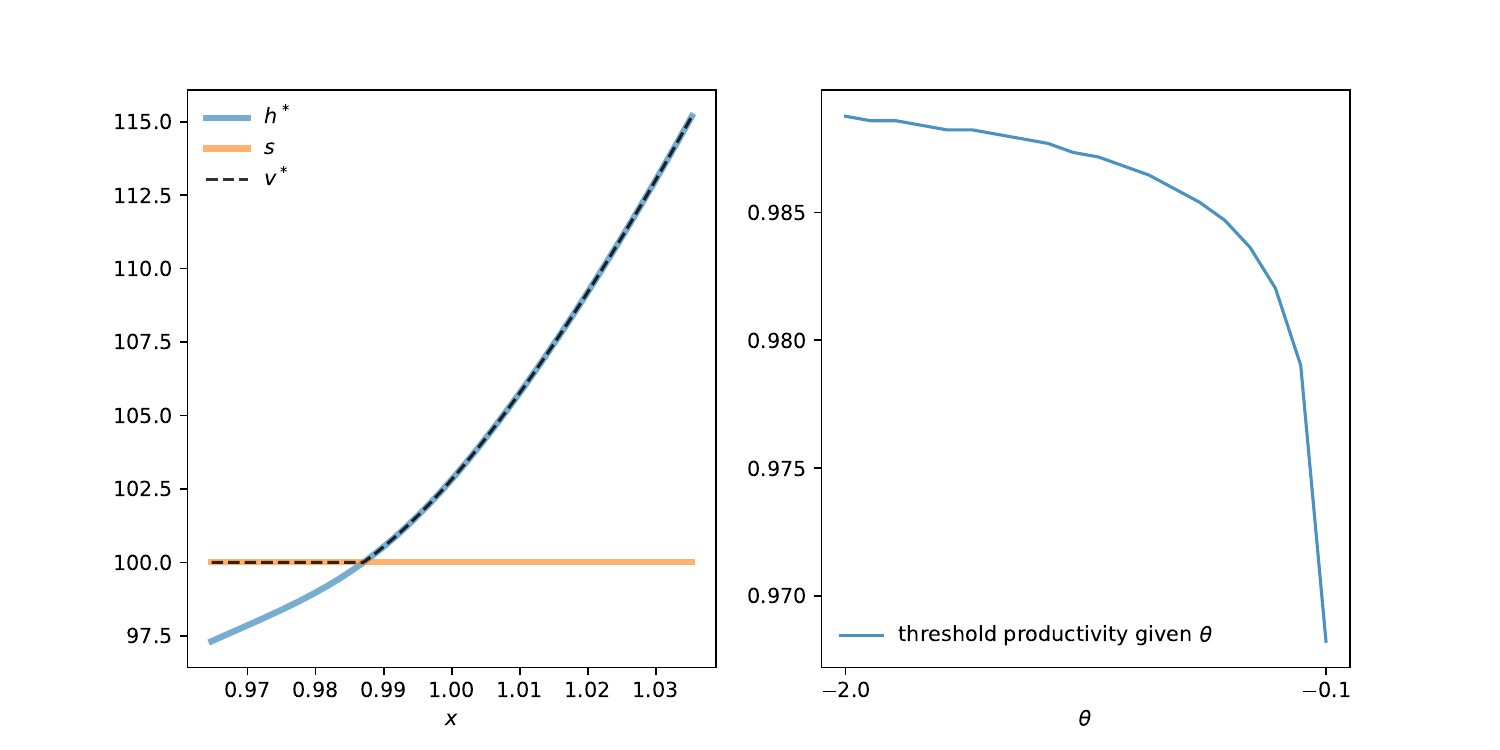}
    \caption{Solution to the firm valuation problem}
    \label{fig:vs}
\end{figure}

\begin{figure}[h]
    \centering
    \includegraphics[width=0.8\linewidth]{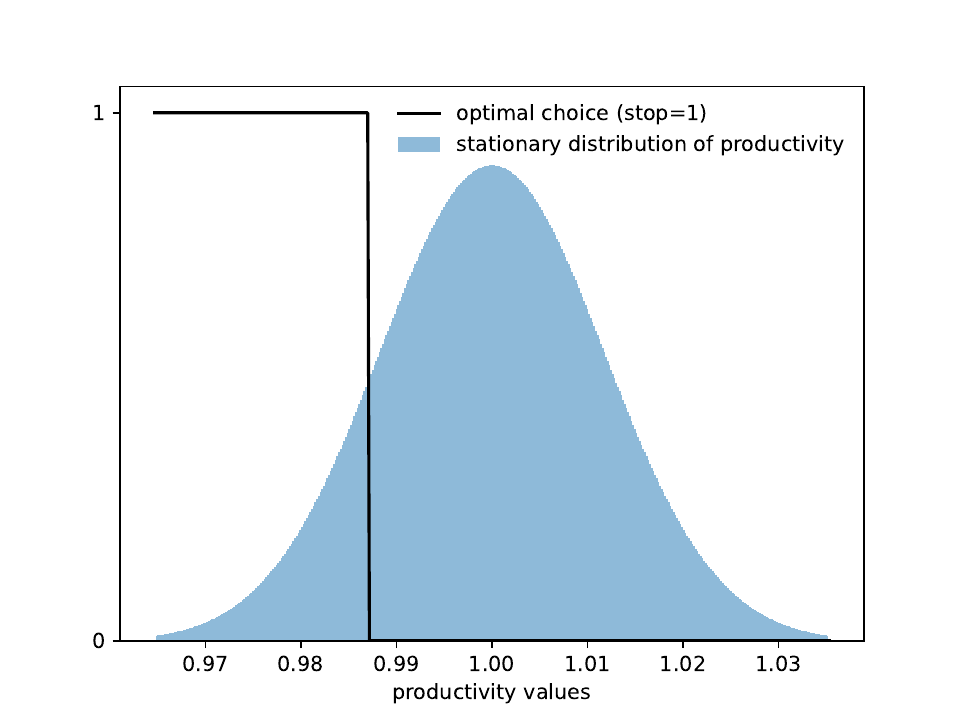}
    \caption{Optimal choice against stationary distribution
    of productivity process}
    \label{fig:policy}
\end{figure}

Figure~\ref{fig:algos} shows the first three iterates of HPI, OPI and VFI, as well as the 
value function $v^*$ and the shared initial condition $v$.  Parameter values are the same
as the left subfigure in Figure~\ref{fig:vs}. In the case of OPI, $m$ is set to 10. We see
that HPI converges faster than VFI in terms of reduced distance to the value
function per iteration. The rate of convergence for OPI is also faster than the
rate for VFI. Figure~\ref{fig:timings} shows execution time for HPI, OPI and VFI
at different choices of $m$. HPI and VFI do not depend on $m$ and hence their
timings are constant.  The figure shows that HPI is faster than VFI in terms of
execution time and that OPI is even faster for some values of $m$.

\begin{figure}[h]
    \centering
    \includegraphics[width=0.8\linewidth]{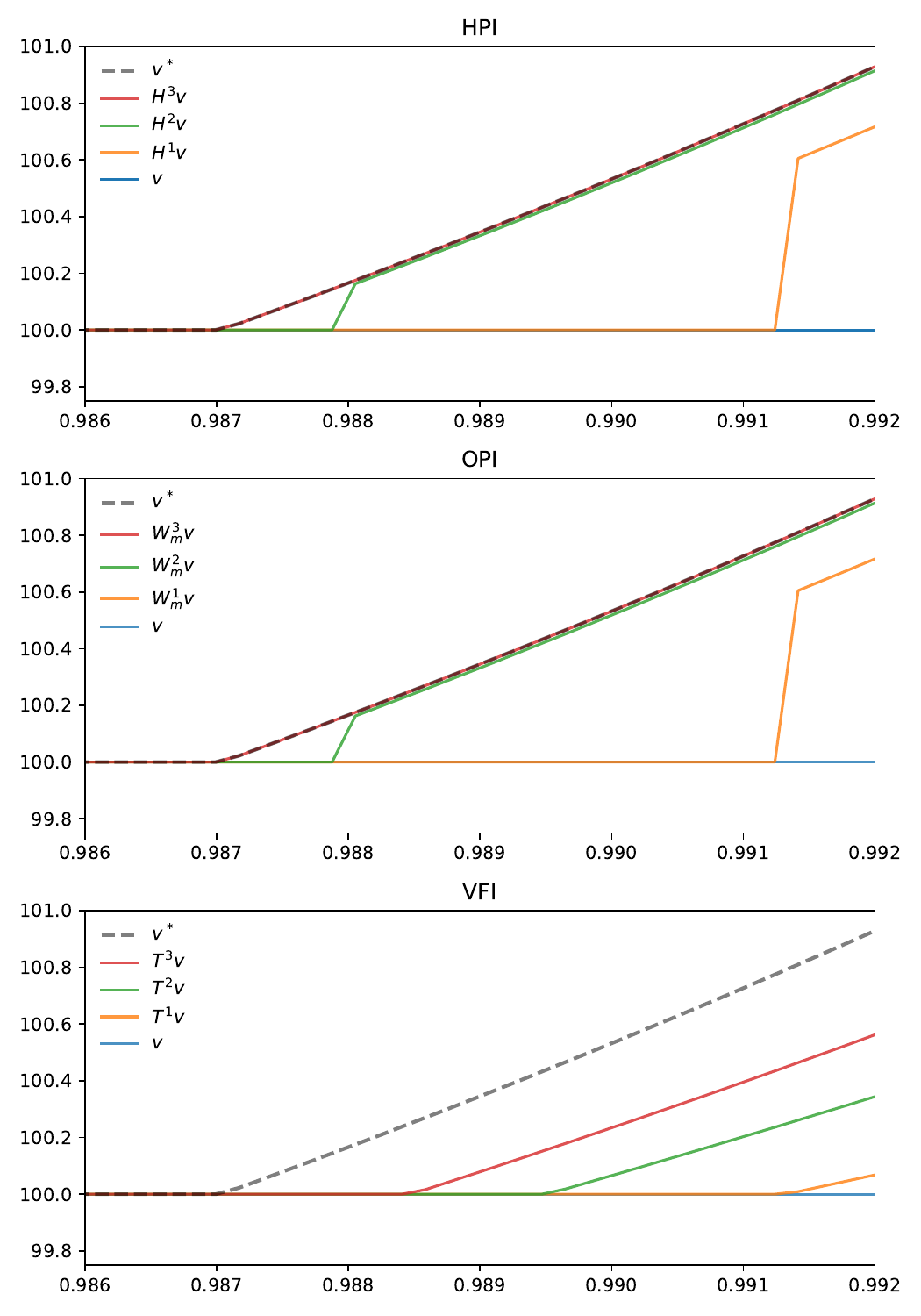}
    \caption{Comparison of algorithms (firm valuation)}
    \label{fig:algos}
\end{figure}

\begin{figure}[h]
    \centering
    \includegraphics[width=\linewidth]{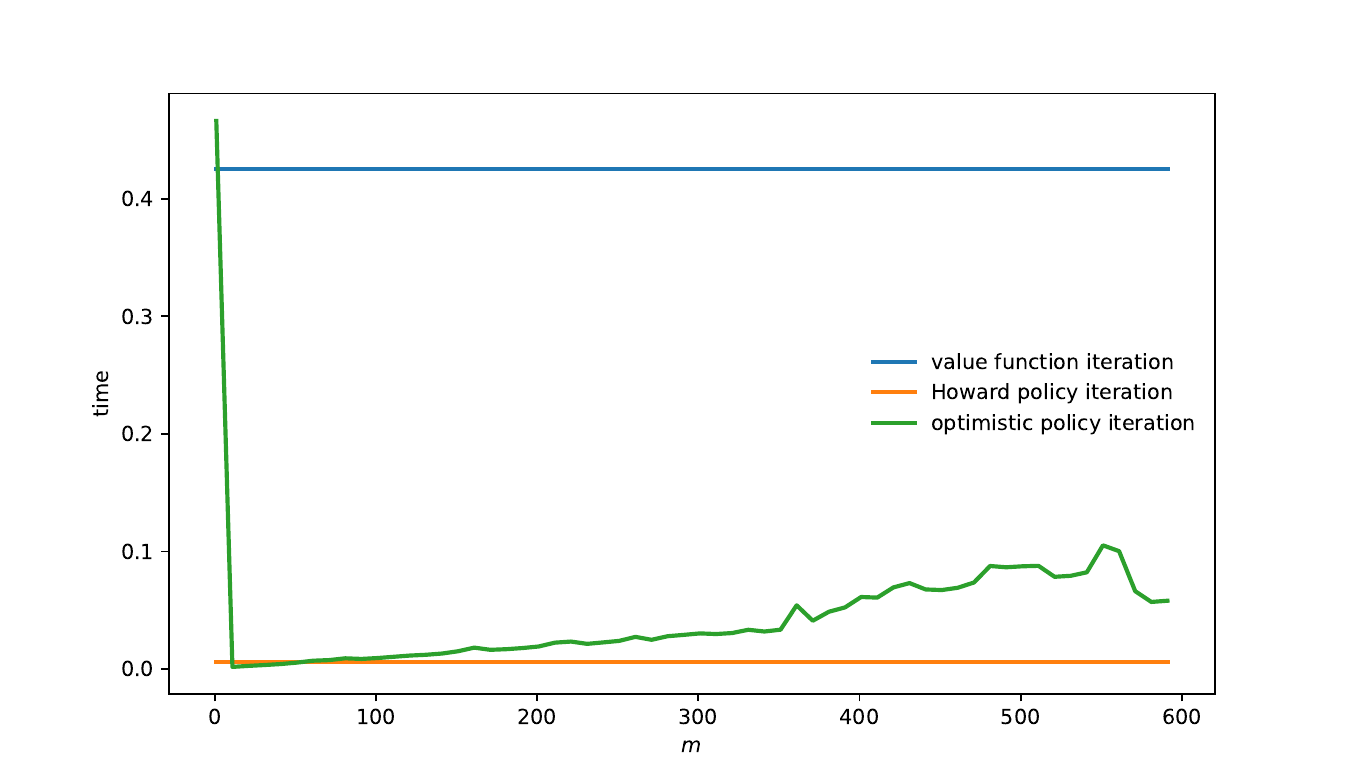}
    \caption{HPI, OPI vs VFI timings for the firm valuation problem}
    \label{fig:timings}
\end{figure}

\subsection{Quantile $Q$-Learning}\label{ss:quantile}

This section gives an application of Theorem~\ref{t:rieszcadp}.
Consider a decision maker with quantile preferences as described 
in \cite{giovannetti2013asset},
\cite{de2019dynamic}, \cite{de2022static} and \cite{TE20250353}. 
Future rewards are discounted at a constant discount factor $\beta \in (0, 1)$.
Let $\Xsf$ and $\Asf$ be finite.
Given $v \in \RR^{\Xsf}$, we set the $Q$-factor corresponding to $v$ as
\begin{equation}\label{eq:q}
    q(x, a) = r(x, a) + \beta \min \left\{
        k \in \RR \colon
        \sum_{z' \in \Zsf} \1 \{ v(H(y, a, z'), z') \leq k\} P(z, z') \geq \tau
    \right\}
\end{equation}
for all $(x, a) \in \Gsf$, where $\tau \in (0, 1)$ parameterizes attitude to risk,
$y$ and $z$ are determined by $x = (y, z)$.

The Bellman equation corresponding to the $Q$-factor quantile decision process
can be written as $v(x) = \max_{a \in \Gamma(x)} q(x, a)$.  Using this
expression in~\eqref{eq:q} yields the \navy{$Q$-factor Bellman equation}
\begin{equation}\label{eq:bmq}
    q(x, a) = r(x, a) + \beta \min \left\{
        k \in \RR \colon
        \sum_{z' \in \Zsf} \1 \{ \max_{a' \in \Gamma(x')} q(x', a') \leq k\}
        P(z, z') \geq \tau
    \right\},
\end{equation}
where $x = (y, z)$ and $x' = (H(y, a, z'), z')$.

To solve~\eqref{eq:bmq}, we introduce an ADP by selecting
$V \coloneq [\underline{r}/(1-\beta), \bar r/(1-\beta)] \subset \RR^{\Gsf}$ as 
the value space where $\underline{r}$ is the infimum of $r$ and $\bar r$ 
is the supremum of $r$.
A feasible policy is a map $\sigma \colon \Xsf \to \Asf$ such that 
$\sigma(x) \in \Gamma(x)$ for all $x \in \Xsf$.
Let $\Sigma$ be the set of all feasible policies.
We set $\TT_{Q} 
\coloneq \{ S_\sigma \colon \sigma \in \Sigma\}$, with each $S_\sigma$
as defined at each $q \in \RR^{\Gsf}$ via
\begin{equation*}
    (S_\sigma q)(x, a) = r(x, a) + \beta \min \left\{
        k \in \RR \colon
        \sum_{z' \in \Zsf} \1 \{ q(x', \sigma(x')) \leq k\} P(z, z') \geq \tau
    \right\}
\end{equation*}
where $x' = (H(y, a, z'), z')$.

\begin{lemma}\label{l:q_radp}
    $(V, \TT_{Q})$ is a regular ADP.
\end{lemma}

\begin{proof}
    The set $\Sigma$ is nonempty since $\Gamma$ is nonempty.  To see that $S_\sigma$
    is order preserving, fix $\sigma \in \Sigma$ and $q \leq f \in \RR^{\Gsf}$, 
    we have $q(x', a') \leq f(x', a')$ for each $(x', a') \in \Gsf$.  
    Hence, given $k \in \RR$ and $x' \in \Xsf$, the probability of 
    $q(x', \sigma(x')) \leq k$ 
    is not less than $f(x', \sigma(x')) \leq k$, we have
    \begin{equation*}
        \1 \{ q(x', \sigma(x')) \leq k\} \geq \1 \{ f(x', \sigma(x')) \leq k\}.
    \end{equation*}
    Thus, given $\tau \in (0, 1)$ and $k \in \RR$, we have
    \begin{equation*}
        \sum_{z' \in \Zsf} \1 \{ f(x', \sigma(x')) \leq k\} P(z, z') \geq \tau 
        \Rightarrow  
        \sum_{z' \in \Zsf} \1 \{ q(x', \sigma(x')) \leq k\} P(z, z') \geq \tau
    \end{equation*}
    for all $(x, a) \in \Gsf$.  Then, $(S_\sigma \; q)(x, a) \leq (S_\sigma \; f)(x, a)$
    since the minimum of the subset is greater.  Hence, $S_\sigma \; q \leq S_\sigma f$.
    
    Regarding $S_\sigma V \subseteq V$, fix $\sigma \in \Sigma$ and 
    $v \in V = [\underline{r}/(1-\beta), \bar r/(1-\beta)]$.  
    Note that for any constant $\lambda$, we have
    \begin{equation*}
        \min \left\{ k \in \RR: \sum_{z' \in \Zsf} \1 \{ \lambda \leq k \}
        P(z, z') \geq \tau
        \right \} = \lambda
    \end{equation*}
    for all $(x, a) \in \Gsf$ and $\tau \in (0, 1)$.  Hence,
    \begin{equation*}
        \underline{r}/(1-\beta) 
        = \underline{r} + \beta \underline{r}/(1-\beta)
        \leq S_\sigma (\underline{r}/(1-\beta)) 
        \leq S_\sigma v,
    \end{equation*}
    and
    \begin{equation*}
        S_\sigma v
        \leq S_\sigma (\bar r/(1-\beta))
        \leq \bar r + \beta \bar r/(1-\beta) = \bar r/(1-\beta).
    \end{equation*}

    Regarding regularity, fix $q \in V$.  There exists a map $\sigma \colon \Xsf
    \to \Asf$ such that $\sigma(x) \in \argmax_{a \in \Gamma(x)} q(x, a)$ 
    for all $x \in \Xsf$, since $\Asf$ is finite.  
    Then $q(x, \sigma(x)) = \max_{a \in \Gamma(x)} q(x, a)$ for all $x \in \Xsf$.
    For this $\sigma$ and any other $\tau \in \Sigma$, we have $(S_\tau \; q)(x, a) \leq
    (S_\sigma \; q)(x, a)$ for each $(x, a) \in \Gsf$, which implies that $S_\tau \; q
    \leq S_\sigma \; q$. In particular, $\sigma$ is $q$-greedy.
\end{proof}

\begin{lemma}\label{l:q_c}
    Each $S_\sigma \in \TT_Q$ satisfies $|S_\sigma^n v - S_\sigma^n w| \too 0$
    for any $v, w \in V$.
\end{lemma}

\begin{proof}
    Fix $\sigma \in \Sigma$.  For any given constant $\lambda \geq 0$ and 
    $q \in \RR^{\Gsf}$,  we have
    \begin{equation*}
        \1 \{ q(x', \sigma(x')) + \lambda \leq k\} 
        = \1 \{ q(x', \sigma(x')) \leq k - \lambda \}
        \text{ for all } x' \in \Xsf \text{ and } k \in \RR.
    \end{equation*}
    Hence, for all $(x, a) \in \Gsf$,
    \begin{align*}
        & (S_\sigma (q+\lambda))(x, a) \\
        = & r(x, a) + \beta \min \left\{
        k \in \RR \colon
        \sum_{z' \in \Zsf} \1 \{ q(x', \sigma(x')) \leq k - \lambda \} P(z, z') 
        \geq \tau \right\} \\
        = & r(x, a) + \beta \min \left\{
        k - \lambda \in \RR \colon
        \sum_{z' \in \Zsf} \1 \{ q(x', \sigma(x')) \leq k - \lambda \} P(z, z')
        \geq \tau \right\} + \beta \lambda \\
        = & (S_\sigma \; q)(x, a) + \beta \lambda,
    \end{align*}
    which implies $S_\sigma \; (q + \lambda) = S_\sigma \; q +  \beta \lambda$.  
    Thus, for any $q, f \in \RR^{\Gsf}$, we have
    \begin{align*}
        S_\sigma \; q & = S_\sigma (q - f + f) \\
        & \leq S_\sigma (f + |q - f|) 
        \leq S_\sigma (f + \sup_{(x, a) \in \Gsf} |q - f|) 
        = S_\sigma \; f + \beta \sup_{(x, a) \in \Gsf} |q - f|.
    \end{align*}
    Reversing the role of $q$ and $f$ yields $|S_\sigma \; q - S_\sigma \; f| \leq \beta
    \sup_{(x, a) \in \Gsf} |q - f|$.  We claim that
    \begin{equation}\label{eq:induction}
        |S_\sigma^n \; q - S_\sigma^n \; f | \leq  \beta^n \sup_{(x, a) \in \Gsf} |q - f|
    \end{equation}
    holds for all $n \in \NN$.  It holds when $n = 1$. Suppose it holds for $n \in \NN$.
    For $n+1$, we have
    \begin{equation*}
        |S_\sigma^{n+1} \; q - S_\sigma^{n+1} \; f |
        \leq \beta \sup_{(x, a) \in \Gsf} |S_\sigma^{n} \; q - S_\sigma^{n} \; f |
        \leq \beta \, \beta^n \sup_{(x, a) \in \Gsf} |q - f|
        = \beta^{n+1} \sup_{(x, a) \in \Gsf} |q - f|.
    \end{equation*}
    The second inequality follows from taking supremum over $\Gsf$ in \eqref{eq:induction}.
    Hence, we have
    \begin{equation*}
        |S_\sigma^n \; q - S_\sigma^n \; f | 
        \leq  \beta^n \sup_{(x, a) \in \Gsf} |q - f| \downarrow 0.
        \qedhere
    \end{equation*}
\end{proof}

\begin{proposition}
    The fundamental optimality properties hold for $(V, \TT_{Q})$ and 
    VFI, OPI and HPI all converge.
\end{proposition}

\begin{proof}
    We showed in Lemma~\ref{l:q_radp} that $(V, \TT_{Q})$ is a regular ADP.
    The value space $V$ is chain complete since $\RR^{\Gsf}$ is Dedekind complete.
    In Lemma~\ref{l:q_c}, we proved that each policy operator $S_\sigma \in \TT_Q$ 
    satisfies $|S_\sigma^n v - S_\sigma^n w| \too 0$ for any $v, w \in V$.
    The claim now follows from Theorem~\ref{t:rieszcadp} (see Remark~\ref{rmk:cc}).
\end{proof}

\subsection{Nonlinear Discounting}\label{ss:nd}

\cite{jaskiewicz2014variable} and \cite{bauerle2021stochastic}
introduced a valuable variation on the standard constant discount factor Markov
decision process model where discounting is allowed to be a nonlinear function
of future lifetime values.  This nonlinearity allows researchers to handle
important features such as magnitude effects, which have been observed
in experimental data \citep{jaskiewicz2014variable}.
We consider a variation of their nonlinear discounting model using a
nonlinear discounting operator and establish
optimality and convergence results using Theorem~\ref{t:rieszbk}. 
In doing so, we suppose that the Bellman equation can be written as
\begin{equation}\label{eq:bend_infinite}
    v(x) = \max_{a \in \Gamma(x)} \left\{
        r(x, a) + \int (\beta v)(H(y, a, z'), z') P(z, \diff z')
    \right\},
    \quad (x=(y,z) \in \Xsf)
\end{equation}
where $\beta$ is an operator that computes future values.
(In the case of \cite{jaskiewicz2014variable} and \cite{bauerle2021stochastic},
the operator $\beta$ has the specific form $\beta v = b \circ v$ for some 
function $b \colon \RR \to \RR_+$.)

We impose some stability conditions.

\begin{assumption}\label{a:nd_infinite}
    The following conditions hold:
    \begin{enumerate}
        \item the map $\bar H(y, z') \coloneq \sup_{a \in \Asf} H(y, a, z')$ 
            mapping from $\Xsf$ to $ \Ysf$ is well-defined and measurable, and 
            for any given $z' \in \Zsf$, it is order-preserving on $\Ysf$,
        \item the discount operator $\beta$ is an order preserving self-map on $bm\Xsf$,
        \item the discount operator $\beta$ obeys 
            $\beta (v + w) \leq \beta v + \beta w$
            for all $v, w \in bm\Xsf$,
        \item there exists a bounded and measurable function 
            $\delta \colon \Zsf \to \RR_+$ such that
            \begin{equation*}
                (\beta v)(y, z) \leq \delta(z) v(y, z) 
                \text{ for all } v \in (bm\Xsf)_+ \text{ and } (y, z) \in \Xsf,
            \end{equation*}
        \item the sequence
            \begin{equation*}
                \sup_{z \in \Zsf} \EE^{P}_z \prod_{t=1}^{n} \delta(Z_t) 
                \downarrow 0.
            \end{equation*}
    \end{enumerate}
\end{assumption}

In (v), $\EE^{P}$ is the conditional expectation 
with respect to the stochastic kernel $P$.

To solve \eqref{eq:bend_infinite}, we construct an ADP in a Riesz space $bm\Xsf$.
Let $S$ be an operator defined at each $v \in (bm\Xsf)_+$ by
\begin{equation*}
    (Sv)(x) 
    = \sup_{(x, a) \in \Gsf} r(x, a) 
    + \int \delta(z') v(\bar H(y, z'), z') P(z, \diff z')
    \text{ for all } x \in \Xsf.
\end{equation*}
The operator $S$ is globally stable on $(bm\Xsf)_+$, by (v) of
Assumption~\ref{a:nd_infinite} and Theorem~2.1 of \cite{stachurski2021dynamic}.  
Let $b \in (bm\Xsf)_+$ be the unique fixed point of $S$.
By (i) of Assumption~\ref{a:nd_infinite} and Theorem~2.2 of \cite{stachurski2021dynamic},
$b$ is increasing in $\Ysf$. We set $V_b = [0, b] \subset bm\Xsf$ as the value space.
A feasible policy is a Borel measurable map $\sigma \colon \Xsf \to \Asf$
such that $\sigma(x) \in \Gamma(x)$ for all $x \in \Xsf$.
Let $\Sigma$ be the set of all feasible policies.
We set $T_{ND} = \{ T_\sigma \colon \sigma \in \Sigma\}$ with
$T_\sigma$ defined as 
\begin{equation*}
    (T_\sigma \; v)(x) 
    = r(x, \sigma(x)) 
    + \int (\beta v)(H(y, \sigma(x), z'), z') P(z, \diff z')
    \text{ for all } x \in \Xsf.
\end{equation*}

\begin{lemma}\label{l:nd_radp}
    If Assumption~\ref{a:regular} and Assumption~\ref{a:nd_infinite} hold, 
    then $(V_b, \TT_{ND})$ is a regular ADP.
\end{lemma}

\begin{proof}
    Fix $\sigma \in \Sigma$. Clearly, $T_\sigma$ is order preserving.  
    For any $v \in V_b$, $\beta v \in bm\Xsf$.  
    The arithmetic and integral operation in the definition of $T_\sigma$
    all preserve boundedness and measurability. 
    By (iv) of Assumption~\ref{a:nd_infinite}, 
    \begin{equation*}
        (T_\sigma b)(x)
        \leq \sup_{(x, a) \in \Gsf} r(x, a) 
        + \int \delta(z') b(H(y, \sigma(x), z'), z') P(z, \diff z').
    \end{equation*}
    This yields $(T_\sigma b)(x) \leq (S b)(x)$, since $b$ is increasing in $\Ysf$.
    Hence, $0 \leq T_\sigma v \leq T_\sigma b \leq S b = b$, thus, $T_\sigma v \in V_b$.

    Regarding regularity, fix $v \in V_b$, we have $\beta v \in V_b$. 
    By (ii) and (iii) of Assumption~\ref{a:regular}, the map
    \begin{equation*}
        (x, a) \mapsto r(x, a) + \int (\beta \; v)(H(y, a, z'), z') P(z, \diff z')
    \end{equation*}
    is measurable on $\Xsf$ and continuous on $\Gsf$, 
    so, by (i) of Assumption~\ref{a:regular} and the measurable maximum theorem
    (see Theorem~18.19 of \cite{aliprantis2006infinite}),
    there exists a measurable function $\sigma \colon \Xsf \to \Asf$ such that
    \begin{equation*}
        \sigma(x) \in \argmax_{a \in \Gamma(x)} \left\{
            r(x, a) + \int (\beta \; v)(H(y, a, z'), z') P(z, \diff z')
        \right\}
        \text{ for all } x \in \Xsf.
    \end{equation*}
    For this $\sigma$ and any other $\tau \in \Sigma$, we have 
    $(T_\tau v)(x) \leq (T_\sigma v)(x)$ for all $x \in \Xsf$.  
    In particular, $\sigma$ is $v$-greedy.
\end{proof}

\begin{lemma}\label{l:nd_aoc}
    If Assumption~\ref{a:nd_infinite} holds, then $(V_b, \TT_{ND})$
    is absolutely order contracting.
\end{lemma}

\begin{proof}
    By (ii) and (iii) of Assumption~\ref{a:nd_infinite}, for any $v, w \in V_b$, we have
    \begin{equation*}
        \beta v = \beta (v - w + w) 
        \leq \beta (|v - w| + w) 
        \leq \beta w + \beta(|v - w|).
    \end{equation*}
    Reversing the role of $v$ and $w$ yields 
    $|\beta v - \beta w| \leq \beta(|v - w|)$.  Hence,
    by (iv) of Assumption~\ref{a:nd_infinite}, 
    for each $\sigma \in \Sigma$ and $x \in \Xsf$, letting $x' = (H(y, \sigma(x), z')$,
    we have
    \begin{align*}
        |(T_\sigma v)(x) - (T_\sigma w)(x)| 
        & = | \int (\beta v)(x') - (\beta w)(x') P(z, \diff z') | \\
        & \leq \int |(\beta v)(x') - (\beta w)(x')| P(z, \diff z') \\
        & \leq \int \beta (|v(x') - w(x')|) P(z, \diff z') \\
        &\leq \int \delta(z') |v(x') - w(x')| P(z, \diff z').
    \end{align*}
    Let a positive linear operator $K$ defined at each $v \in V_b$ by
    \begin{equation*}
        (Kv)(x) = \int \delta(z') v(x') P(z, \diff z')
        \text{ for all } x \in \Xsf,
    \end{equation*}
    Clearly, $K$ is order continuous.  Fix $v_n \uparrow v$ in $V$.  
    By Lemma~\ref{l:bmx}, we have $\delta(z') v_n(x') \uparrow \delta(z') v(x')$.
    Fixing $x \in \Xsf$ and applying the
    dominated convergence theorem yields $(K v_n)(x) \uparrow (K v)(x)$.  
    Applying Lemma~\ref{l:bmx} again yields $K v_n \uparrow K v$.

    Moreover, by (v) of Assumption~\ref{a:nd_infinite}, 
    for any $v \in V_b$ and $x \in \Xsf$, 
    \begin{equation*}
        (K^n v)(x) 
        \leq \sup_{z \in \Zsf} \EE^{P}_z \left[\prod_{t = 1}^{n} \delta(Z_t) \right] 
        \sup_{x \in \Xsf} v(x) \downarrow 0.
    \end{equation*}
    The inequality follows from Section A.4.1 of \cite{stachurski2021dynamic}. 
    Hence, $T_\sigma$ is absolutely order contracting.
\end{proof}

\begin{proposition}
    If Assumption~\ref{a:regular} and Assumption~\ref{a:nd_infinite} hold, then
    the fundamental optimality properties hold for $(V_b, \TT_{ND})$ and
    VFI, OPI and HPI all converge.
\end{proposition}

\begin{proof}
    We showed in Lemma~\ref{l:nd_radp} that $(V_b, \TT_{ND})$ is a regular ADP. 
    The value space $V_b$ is countably chain complete 
    since $bm\Xsf$ is countably Dedekind complete.
    In Lemma~\ref{l:nd_aoc}, we proved that $(V_b, \TT_{ND})$ 
    is absolutely order contracting.
    The claims now follow from Theorem~\ref{t:rieszbk}.
\end{proof}

\section{Conclusion}

This paper presents a unified framework for dynamic programming within ordered vector
spaces. On the theoretical side, we first established sufficient conditions for the existence
and uniqueness of fixed points under concavity and order-contractivity, and then 
demonstrated how these fixed-point results can be used to obtain optimality
properties and algorithmic convergence results. On the practical side, we showed how the theory
applies to a broad class of decision-making problems, such as firm management.  Our examples discussed
risk-sensitive preferences with state-dependent discounting, quantile dynamic programs
linking with Q-learning formulations, and nonlinear discounting models.  Together, these
results underscore the fact that order-theoretic methods provide a natural mathematical
environment for modern dynamic programming.

\theendnotes

\bibliographystyle{ecta}
\bibliography{localbib}

\end{document}